\documentclass[a4paper]{article}
\usepackage[utf8]{inputenc}
\usepackage{graphicx}
\usepackage{flushend}
\usepackage{amsmath}
\usepackage{amsthm}
\usepackage{nccmath}
\usepackage{amssymb}
\usepackage{enumerate}
\usepackage[utf8]{inputenc}
\usepackage{csquotes}

\usepackage{natbib}

\usepackage[off]{auto-pst-pdf}
\usepackage{ifplatform,pst-pdf,xkeyval}
\usepackage{pstricks}
\usepackage{pst-all}
\usepackage{pst-plot}
\usepackage{pstricks-add}
\usepackage{pst-node}
\usepackage{pst-ode}
\usepackage{pst-3dplot}
\usepackage{caption}
\usepackage{pst-circ}
\usepackage{float}
\usepackage{pst-lens}
\usepackage{scalerel}

\usepackage{hyperref}
\hypersetup{pdftex,colorlinks=true,allcolors=blue}
\usepackage{hypcap}

\def\r{\mathbb{R}}

\theoremstyle{plain}
 \newtheorem{theorem}{Theorem}[section]
 \newtheorem{definition}[theorem]{Definition}
 
 \newtheorem{remark}[theorem]{Remark}
 \newtheorem{lemma}[theorem]{Lemma}
 \newtheorem{corollary}[theorem]{Corollary}
 \newtheorem{claim}{Claim}

\providecommand{\keywords}[1]{
    \small\textbf{\textit{Keywords---}} #1
}

\begin{document}
\title{Elliptic special Weingarten surfaces of minimal type in $\r ^3$ of finite total curvature}
\author{José M. Espinar$^{\dag}$ and Héber Mesa$^{\ddag}$}
\date{}
\maketitle

\begin{center}
{\footnotesize $^{\dag}$Departamento de Geometría y Topología, Universidad de Granada,\\ Granada, Spain\\Email: jespinar@ugr.es\\https://orcid.org/0000-0003-1323-6648\\}
\vspace{.3cm}
{\footnotesize $^{\ddag}$Departamento de Matem\'aticas, Universidad del Valle,\\ Santiago de Cali, Colombia\\ 
Email: heber.mesa@correounivalle.edu.co\\https://orcid.org/0000-0002-9952-575X\\}
\end{center}

\begin{abstract}
We study complete connected embedded elliptic special Weingarten surfaces of minimal type (\emph{$f$-surfaces}, for short) in $\mathbb{R}^3$ with finite total curvature, under the standing assumption that $f$ is a non-negative and uniformly elliptic function. First, using the recent asymptotic theory of Barbieri, Gálvez, Lian, and Zhang for embedded ends, we derive a Jorge--Meeks type formula for this class of surfaces. Next, we adapt the Alexandrov reflection method to the non-cylindrically bounded setting and prove a Schoen-type theorem: a complete connected embedded $f$-surface of finite total curvature with two embedded ends must be rotationally symmetric. In particular, it is one of the special catenoids constructed by Sa Earp and Toubiana. This gives a positive answer to a question raised by Sa Earp in 1993. As a consequence, we show that planes and special catenoids are the only complete connected embedded $f$-surfaces of finite total curvature whose absolute total curvature is less than $8\pi$.
\end{abstract}

\keywords{minimal surface, elliptic Weingarten surface, finite total curvature, geometric maximum principle.}

\section{Introduction}
\addcontentsline{toc}{section}{Introduction}

Complete minimal surfaces in $\r^3$ with finite total curvature form one of the classical themes of global differential geometry. By the work of Huber and Osserman, such surfaces are conformally equivalent to compact Riemann surfaces with finitely many punctures, and their Gauss map extends meromorphically across the ends; see \cite{Huber1,Osserman1}. In the seminal paper of Jorge and Meeks \cite{Jorge-Meeks}, the total curvature was related to the topology of the surface through their celebrated formula, and the structure of embedded ends was analyzed. These ideas were later used by Schoen \cite{Schoen1} to prove that the catenoid is the only complete connected immersed minimal surface in $\r^3$ with finite total curvature and two embedded ends. As a consequence, the plane and the catenoid are the only complete embedded minimal surfaces of finite total curvature whose absolute total curvature is less than $8\pi$.

In this paper we consider a broader class of surfaces in $\r^3$ that includes minimal surfaces, namely elliptic special Weingarten surfaces, abbreviated as ESW-surfaces. These surfaces satisfy a relation of the form
\[
H=f(H^2-K),
\]
where $H$ and $K$ denote the mean and Gaussian curvatures, respectively, and $f\in C^0([0,+\infty))\cap C^1((0,+\infty))$ satisfies the ellipticity condition
\[
4t f'(t)^2<1 \qquad \text{for all } t> 0.
\]
We will work under slightly stronger standing assumptions, namely that $f$ is non-negative and that the ellipticity condition is uniform; see \eqref{elliptic-conditions} below.
From the geometric point of view, the theory of elliptic Weingarten surfaces may be regarded as a fully nonlinear counterpart of the theory of constant mean curvature surfaces.

Closed Weingarten surfaces in $\r^3$ were extensively studied by Hartman and Wintner \cite{Hartman}, Chern \cite{Chern3,Chern2}, Hopf \cite{HHopf}, and Bryant \cite{Bryant}, who obtained, among other results, generalizations of Hopf's theorem and Liebmann's theorem. In the non-compact setting, however, additional assumptions on the function $f$ near $0$ and/or at infinity are needed in order to rule out pathological behavior.

A decisive step in the global theory was made by Rosenberg and Sa Earp \cite{Rosenberg-SaEarp1}, who distinguished between two classes of ESW-surfaces: those of constant mean curvature type, characterized by the condition $f(0)\neq 0$, and those of minimal type, characterized by $f(0)=0$. This distinction is geometrically meaningful. In the constant mean curvature type case, the sphere of radius $1/|f(0)|$ belongs to the family, and Rosenberg and Sa Earp were able to extend the Korevaar--Kusner--Meeks--Solomon theory \cite{KKS,Meeks} for properly embedded constant mean curvature surfaces in $\r^3$ to this Weingarten setting, under additional assumptions on $f$. In particular, when $H$ and $K$ satisfy a linear relation
\[
aH+bK=c,
\]
the corresponding surfaces are called linear Weingarten surfaces, and Rosenberg and Sa Earp \cite{Rosenberg-SaEarp1} proved that, if $b=1$ and $a,c>0$, the annular ends of a properly embedded linear Weingarten surface converge to Delaunay ends. Additional results in this direction can be found in \cite{Braga-SaEarp1,GMM,SaEarp4,SaEarp3}.

The minimal type case is the one relevant for the present paper. In a series of papers \cite{SaEarp2,SaEarp1}, Sa Earp and Toubiana constructed rotationally symmetric examples, in particular the special catenoids, which play for this theory a role analogous to that of the classical catenoid in the minimal setting. These examples allowed them to extend the Hoffman--Meeks half-space theorem \cite{Hoffman-Meeks} to the class of ESW-surfaces under suitable additional assumptions. They also proved that the Gaussian curvature is non-positive, that its zeros are isolated, and that the convex hull property holds, all of which are features shared with minimal surfaces.

Nearly twenty years later, Aledo, Espinar, and Gálvez \cite{AEG} extended the Korevaar--Kusner--Meeks--Solomon theory without any additional assumptions on $f$. They also classified ESW-surfaces whose Gaussian curvature does not change sign, obtaining results analogous to those of Klotz and Osserman \cite{Klotz-Osserman} for constant mean curvature surfaces. More precisely, they proved that a complete ESW-surface with $K\geq 0$ must be either a totally umbilical sphere, a plane, or a right circular cylinder. They also showed that a properly embedded ESW-surface with $K\leq 0$ is either a right circular cylinder or satisfies $f(0)=0$, that is, it is of minimal type. Later, Gálvez, Martínez, and Teruel \cite{GMT} improved this result for the case $K\leq 0$ by replacing proper embeddedness with completeness. More recently, Fernández, Gálvez, and Mira \cite{FGM} proved Bernstein-type theorems for elliptic Weingarten surfaces, while Fernández and Mira \cite{FerMira} revisited the rotational theory and the half-space phenomenon from a broader perspective.

Despite this progress, the finite-total-curvature theory for ESW-surfaces of minimal type has remained substantially less developed than in the minimal case. The main reason is that one no longer has a Weierstrass representation, and therefore the asymptotic behavior of the ends must be extracted directly from a fully nonlinear elliptic equation. In particular, the analysis of embedded ends of finite total curvature is much more delicate than in the classical minimal setting.

A decisive new input in this direction is the recent work of Barbieri, Gálvez, Lian, and Zhang \cite{BGLZ26}, where the asymptotic expansion at infinity of embedded ends of finite total curvature is established for uniformly elliptic Weingarten surfaces of minimal type, together with a maximum principle at infinity. These results provide the asymptotic information needed in our setting and allow us to combine the finite-total-curvature theory with the Alexandrov reflection method in the non-cylindrically bounded case.

The main purpose of this paper is to extend to elliptic special Weingarten surfaces of minimal type the finite-total-curvature theory that, in the minimal case, is encoded in the classical works of Jorge--Meeks and Schoen. More precisely, we prove that a complete connected embedded $f$-surface of finite total curvature with two embedded ends must be rotationally symmetric; in particular, it is one of the special catenoids constructed by Sa Earp and Toubiana. This gives a positive answer to a question posed by Sa Earp in 1993; see \cite{SaEarp2}. As an application, we prove that special catenoids and affine planes are the only complete connected embedded ESW-surfaces of minimal type whose absolute total curvature is less than $8\pi$. 

\subsection*{Organization of the paper}

The paper is organized as follows. In Section \ref{SectPre}, we recall the basic properties of ESW-surfaces of minimal type, review the rotational examples that will be used later, and collect the asymptotic information on embedded ends of finite total curvature that follows from \cite{BGLZ26}. In Section \ref{SectJM}, we prove a Jorge--Meeks type formula for complete embedded ESW-surfaces of minimal type and finite total curvature. Section \ref{SectAlexandrov} is devoted to the Alexandrov reflection method for catenoidal-type ends. After introducing the relevant Alexandrov functions, we adapt the method to our non-cylindrically bounded setting by means of tilted planes and prove a Schoen-type theorem. Finally, in Section \ref{SectApJM}, we combine the previous results to classify complete connected embedded ESW-surfaces of minimal type and finite total curvature whose absolute total curvature is less than $8\pi$.

\section{Preliminaries}\label{SectPre}

Let $\Sigma$ denote a connected oriented immersed surface in $\mathbb{R}^3$, possibly with boundary, and let $N$ be its unit normal. We write $I$, $II$, $H$, $K$, $q$, $k_1$, and $k_2$ for the first fundamental form, the second fundamental form, the mean curvature, the Gaussian curvature, the skew curvature, and the principal curvatures, respectively. Thus,
\[
2H=k_1+k_2,\qquad K=k_1k_2,\qquad q=H^2-K.
\]

\subsection{Elliptic special Weingarten surfaces of minimal type}

An oriented immersed surface $\Sigma\subset \mathbb{R}^3$ is called an \emph{elliptic special Weingarten surface of minimal type}, or simply an \emph{$f$-surface}, if its mean curvature and Gaussian curvature satisfy
\[
H=f(H^2-K),
\]
where $f\in C^0([0,+\infty))\cap C^1((0,+\infty))$, $f(0)=0$, and
\[
4t f'(t)^2<1\qquad \text{for all } t> 0.
\]
In this case the associated equation is elliptic, and the geometric maximum principle is available; see \cite{Braga-SaEarp1,Rosenberg-SaEarp1}.

\begin{definition}\label{def:tangent-point}
Let $\Sigma_1$ and $\Sigma_2$ be immersed surfaces in $\mathbb{R}^3$, and let $p\in \Sigma_1\cap \Sigma_2$. We say that $p$ is a \emph{tangent point} of $\Sigma_1$ and $\Sigma_2$ if
\[
T_p\Sigma_1=T_p\Sigma_2,\qquad N_1(p)=N_2(p),
\]
and either $p$ is an interior point of both surfaces, or else $p\in \partial\Sigma_1\cap \partial\Sigma_2$ and the interior conormal vectors of $\partial\Sigma_1$ and $\partial\Sigma_2$ coincide at $p$.
\end{definition}

Given an immersed surface $\Sigma\subset \mathbb{R}^3$ and a point $p\in \Sigma$, the surface can be written locally near $p$ as a graph over its tangent plane at $p$. Therefore, if $\Sigma_1$ and $\Sigma_2$ have a common tangent point $p$, then near $p$ they can be expressed as graphs of functions $u_1$ and $u_2$, respectively, over the same domain in the common tangent plane.

We write
\[
\Sigma_1\geq_p \Sigma_2
\]
if $\Sigma_1$ and $\Sigma_2$ have a common tangent point at $p$ and, in a neighborhood of $p$, the corresponding graphing functions satisfy $u_1\geq u_2$. Geometrically, this means that $\Sigma_1$ lies above $\Sigma_2$ near $p$.

\begin{theorem}[Geometric maximum principle, \cite{Braga-SaEarp1,Rosenberg-SaEarp1}]\label{geometric-maximum}
Let $\Sigma_1$ and $\Sigma_2$ be connected $f$-surfaces immersed in $\mathbb{R}^3$, for the same function $f$. If $\Sigma_1\geq_p \Sigma_2$ for some point $p$, then $\Sigma_1$ and $\Sigma_2$ coincide in a neighborhood of $p$. If, in addition, $\Sigma_1$ and $\Sigma_2$ are complete and without boundary, then $\Sigma_1=\Sigma_2$.
\end{theorem}

We will repeatedly use the following standard properties of $f$-surfaces; see \cite{AEG,SaEarp2,SaEarp1}:
\begin{itemize}
\item[(i)] the Gaussian curvature satisfies $K\leq 0$;
\item[(ii)] the zeros of $K$ are isolated unless the surface is a plane;
\item[(iii)] a point is umbilic if and only if $K=0$ at that point;
\item[(iv)] a compact $f$-surface with boundary is contained in the convex hull of its boundary;
\item[(v)] there is no closed $f$-surface in $\mathbb{R}^3$;
\item[(vi)] planes are the only $f$-surfaces with identically zero Gaussian curvature.
\end{itemize}

Throughout the paper we impose the stronger assumptions
\begin{equation}\label{elliptic-conditions}
\left\{
\begin{array}{l}
f\in C^1([0,+\infty)),\\[1mm]
f \text{ is non-negative},\\[1mm]
4t f'(t)^2\leq c<1 \text{ for some } c\in (0,1).
\end{array}
\right.
\end{equation}
These are precisely the hypotheses that will be used throughout the paper.

A word on the role of each condition is in order, since \eqref{elliptic-conditions} is stronger than the pointwise ellipticity $4tf'(t)^2<1$ mentioned in the Introduction. The regularity $f\in C^1([0,+\infty))$ is the smoothness under which the geometric maximum principle (Theorem \ref{geometric-maximum}) and the rotational theory of Sa Earp and Toubiana are available. The \emph{uniform} ellipticity $4tf'(t)^2\leq c<1$ is the standing hypothesis of the asymptotic theory of \cite{BGLZ26}: as we verify below (see the computation preceding Theorem \ref{thm:log-expansion}), it is equivalent to a uniform two-sided bound for the slope of the associated curvature relation, which is exactly the form of ellipticity required in \cite{BGLZ26}; the pointwise condition $4tf'(t)^2<1$ alone does not yield the uniform estimates at infinity on which our analysis of the ends rests. Finally, the sign condition $f\geq 0$, together with $f(0)=0$, places us in the minimal-type setting of Sa Earp and Toubiana: it guarantees the existence and uniqueness of the complete rotational examples $M_\tau$ (Theorem \ref{existence}) and the tangency and half-space principles (Theorem \ref{thm:SaEarp3}, Corollary \ref{thm:half-space}) that underlie the reflection arguments of Section \ref{SectAlexandrov}; see Remark \ref{rem:elliptic-implies-rotational-hypotheses}.

\subsection{Rotationally symmetric examples}

The rotational theory of elliptic Weingarten surfaces was developed in depth by Sa Earp and Toubiana \cite{SaEarp2,SaEarp1}; for a broader and more recent discussion, including singular and non-complete rotational examples, we refer the reader to \cite[Section 5.2]{FerMira}. In the present paper we only use the complete regular rotational examples of minimal type.

\begin{theorem}[Existence and uniqueness of rotational $f$-surfaces, \cite{SaEarp1}]\label{existence}
Let $f$ be an elliptic function satisfying $f(0)=0$ and 
\begin{equation}\label{inf-condition}
\liminf_{t\to 0^+} 4t f'(t)^2<1.
\end{equation}
Let $\tau>0$ satisfy
\begin{equation}\label{SaEarp-condition}
\frac{1}{\tau}<\lim_{t\to\infty}\big(\sqrt{t}-f(t)\big).
\end{equation}
Then there exists a unique (up to an ambient isometry) complete rotational $f$-surface $M_\tau$. Moreover, the generating curve of $M_\tau$ is the graph of a symmetric, strictly positive, convex $C^3$ function whose global minimum equals $\tau$.

Conversely, every complete rotational $f$-surface arises in this way.
\end{theorem}

We will refer to the surfaces $M_\tau$ as \emph{special catenoids}; for $f\equiv 0$ they are the standard catenoids.

From now on, every surface $M_\tau$ is oriented by the \emph{exterior} unit normal, namely the unit normal pointing into the connected component of $\mathbb{R}^3\setminus M_\tau$ that does not contain the rotation axis. With this convention, the mean curvature is the one used in \cite{SaEarp1}.

Depending on the generating curve, there are two possible behaviors: either $M_\tau$ is contained in a slab bounded by two parallel planes, or all its coordinate functions are proper. See Figure \ref{generatrix}. Under the standing assumptions \eqref{elliptic-conditions}, the first behavior cannot occur: a surface contained in a slab lies in a half-space, so Corollary \ref{thm:half-space} would force $M_\tau$ to be a plane, contrary to $M_\tau$ being a non-flat rotational example. Hence every such $M_\tau$ has proper coordinate functions.

\begin{figure}[ht]
\centering
\includegraphics{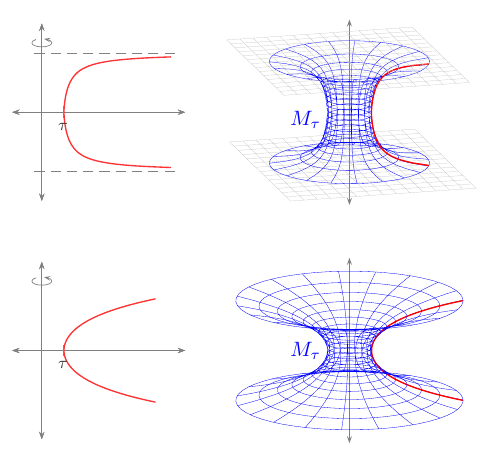}
\caption{Generating curve (left) and complete rotational examples $M_\tau$ (right).}
\label{generatrix}
\end{figure}

\begin{theorem}[Sa Earp--Toubiana, \cite{SaEarp1}]\label{thm:SaEarp2}
Let $f$ be an elliptic function satisfying $f(0)=0$ and \eqref{inf-condition}, and let $\tau>0$ satisfy \eqref{SaEarp-condition}. If $f$ is Lipschitz at $0$, then the surface $M_\tau$ is not asymptotic to any plane in $\mathbb{R}^3$. Moreover, there exists a classical catenoid $C$ such that, outside a compact set, $M_\tau$ lies in the component of $\mathbb{R}^3\setminus C$ containing the rotation axis of $C$.
\end{theorem}

\begin{theorem}[Sa Earp--Toubiana, \cite{SaEarp1}]\label{thm:SaEarp3}
Let $f$ be a non-negative elliptic function satisfying $f(0)=0$ and \eqref{inf-condition}. Assume that $f$ is Lipschitz at $0$ and that
\begin{equation}\label{SaEarp-condition2}
\lim_{t\to\infty}\big(\sqrt{t}-f(t)\big)=+\infty.
\end{equation}
Then, as $\tau\to 0$, the generating curve of $M_\tau$ converges to a ray orthogonal to the rotation axis. Furthermore, $M_\tau$ lies in the component of $\mathbb{R}^3\setminus C_\tau$ that does not contain the rotation axis of $C_\tau$, where $C_\tau$ is the catenoid whose generating curve is
\[
c_\tau(s)=\tau\cosh(\tau^{-1}s).
\]
\end{theorem}

\begin{figure}[ht]
\centering
\includegraphics{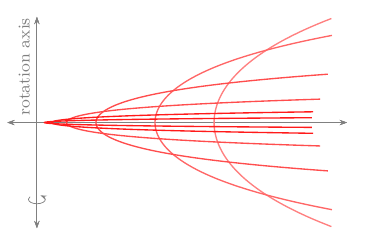}
\caption{Generating curves of rotational $f$-surfaces as $\tau \to 0$.}
\label{fig:generatrices}
\end{figure}

As a consequence, the family $\{M_\tau:\tau>0\}$ provides the catenoidal barriers needed in the half-space argument; see Figure \ref{fig:generatrices}.

\begin{corollary}[Half-space theorem for $f$-surfaces, \cite{SaEarp1}]\label{thm:half-space}
Let $f$ be a non-negative elliptic function, Lipschitz at $0$, satisfying $f(0)=0$, \eqref{inf-condition}, and \eqref{SaEarp-condition2}. Let $\Sigma$ be a complete connected properly immersed $f$-surface in $\mathbb{R}^3$. If $\Sigma$ is contained in a half-space, then $\Sigma$ is a plane.
\end{corollary}

\begin{remark}\label{rem:elliptic-implies-rotational-hypotheses}
The assumptions in \eqref{elliptic-conditions} imply the hypotheses of Theorem \ref{thm:SaEarp3} and Corollary \ref{thm:half-space}. Indeed, since $f\in C^1([0,+\infty))$, the function $f$ is Lipschitz at $0$. Moreover, from $4tf'(t)^2\leq c$ we get $|f'(t)|\leq \sqrt{c}/(2\sqrt{t})$ for $t>0$, and hence, as $f(0)=0$,
\[
f(t)=\int_0^t f'(s)\,ds\leq \sqrt{c}\,\sqrt{t}\qquad \text{for all } t\geq 0 .
\]
Therefore $\sqrt{t}-f(t)\geq (1-\sqrt{c})\sqrt{t}\to+\infty$ as $t\to+\infty$, which is condition \eqref{SaEarp-condition2}; condition \eqref{inf-condition} is immediate.
\end{remark}

\subsection{Finite total curvature}\label{SectFTC}

A complete immersed surface $\Sigma\subset \mathbb{R}^3$ is said to have \emph{finite total curvature} if
\[
\int_\Sigma |K|\, dA<+\infty.
\]

The next result records the structural facts on finite-total-curvature ends that will be used later.

\begin{lemma}\label{LemFinite}
Let $\Sigma\subset \mathbb{R}^3$ be a complete connected immersed $f$-surface satisfying \eqref{elliptic-conditions}, with finite total curvature and embedded ends. Then $\Sigma$ has finite conformal type. Moreover, every end of $\Sigma$ is properly embedded, has a limit unit normal at infinity, and, after possibly truncating the end and applying an isometry of $\mathbb{R}^3$, it can be written as the graph of a function defined over the complement of a compact set in a plane.
\end{lemma}

\begin{proof}
Finite conformal type follows from Huber's theorem \cite{Huber1}. Each end of $\Sigma$ is then an annular surface which, once truncated, is a complete embedded $f$-surface with compact boundary and finite total curvature. The description of the ends is therefore exactly \cite[Lemma 2.1]{BGLZ26}, which is stated for complete embedded surfaces possibly with compact boundary.
\end{proof}

\subsection{Asymptotic behavior of embedded ends}\label{SectAsymptotic}

Let $E$ be an embedded end of a complete connected immersed $f$-surface $\Sigma$ satisfying \eqref{elliptic-conditions} and having finite total curvature. By Lemma \ref{LemFinite}, after a rigid motion and truncation, we may write
\[
E=\{(x,u(x))\in \mathbb{R}^2\times \mathbb{R}: x\in \mathbb{R}^2\setminus K\},
\]
for some smooth function $u$ defined over the complement of a compact set $K\subset \mathbb{R}^2$, and with
\[
\lim_{|x|\to +\infty}N(x)=N_\infty.
\]
After a rotation we may assume
\(
N_\infty=(0,0,1).
\)
The relation
\(
H=f(H^2-K)
\)
defines a symmetric elliptic Weingarten relation in the $(\kappa_1,\kappa_2)$-plane. By \cite[Section 2.1]{FGM}, every connected branch of this relation can be written as
\[
\kappa_2=\varphi(\kappa_1),
\]
where $\varphi$ is decreasing and satisfies
\(
\varphi\circ \varphi={\rm Id}.
\)
Since the relation is symmetric, each connected branch is invariant under the exchange $(\kappa_1,\kappa_2)\mapsto(\kappa_2,\kappa_1)$, so $\varphi$ is a decreasing involution of an interval onto itself and has a unique fixed point $t_0$. At the fixed point, the surface relation gives $t_0=f(0)=0$, since $H=t_0$ and $H^2-K=0$ there. Hence, differentiating $\varphi\circ \varphi={\rm Id}$ we get
\(
\varphi'(0)^2=1.
\)
As $\varphi'<0$, we conclude that
\begin{equation}\label{eq:phi-prime-zero}
\varphi'(0)=-1.
\end{equation}

In fact, the uniform ellipticity assumption in \eqref{elliptic-conditions} translates into a uniform two-sided bound for $\varphi'$, which is the form of ellipticity required in \cite{BGLZ26}. Indeed, writing $\kappa_2=\varphi(\kappa_1)$ with $\kappa_1\geq\kappa_2$, we have $2\sqrt{q}=\kappa_1-\kappa_2$, where $q=H^2-K$, and differentiating $H=f(q)$ along the branch yields
\[
\frac{1+\varphi'}{2}=f'(q)\,\sqrt{q}\,(1-\varphi'),
\qquad\text{that is,}\qquad
\varphi'=\frac{s-1}{s+1},
\qquad s:=2\sqrt{q}\,f'(q).
\]
By \eqref{elliptic-conditions} we have $|s|\leq\sqrt{c}<1$, and since $s\mapsto (s-1)/(s+1)$ is increasing on $(-1,1)$,
\[
-\frac{1}{\Lambda_0}\leq \varphi'(t)\leq -\Lambda_0,
\qquad
\Lambda_0:=\frac{1-\sqrt{c}}{1+\sqrt{c}}\in(0,1),
\]
for every $t$ in the branch. In particular $\varphi$ is locally Lipschitz. Therefore the graphical representation of $E$ is a uniformly elliptic Weingarten graph of minimal type in the sense of \cite{BGLZ26}, and the asymptotic results of \cite{BGLZ26} apply to it. Since $\varphi'(0)=-1$ by \eqref{eq:phi-prime-zero}, the end falls under the case of \cite[Theorem 1.3]{BGLZ26} that yields logarithmic growth at infinity.

\begin{theorem}[Theorems 1.2 and 1.3 in \cite{BGLZ26}]\label{thm:log-expansion}
Let $E$ be as above, and let
\[
E=\{(x,u(x))\in \mathbb{R}^2\times \mathbb{R}: x\in \mathbb{R}^2\setminus K\}
\]
be its graphical representation with
\[
\lim_{|x|\to+\infty}N(x)=(0,0,1).
\]
Then, after a vertical translation, exactly one of the following two possibilities occurs:
\[
u(x)>0 \, \, \text{ or } \, \, u(x)<0 \qquad \text{for all } x\in \mathbb{R}^2\setminus K.
\]

Moreover, the limit
\[
u_\infty:=\lim_{|x|\to+\infty}u(x)
\]
exists in $\mathbb{R}\cup\{\pm\infty\}$. If we assume that $ u(x)>0 $ for all $x\in \mathbb{R}^2\setminus K$, then there exist constants $d\geq 0$ and $c_0\in \mathbb{R}$ such that, for every $0<\alpha<1$,
\begin{equation}\label{eq:log-expansion}
u(x)=d\log |x|+c_0+O(|x|^{-\alpha})
\qquad \text{as } |x|\to +\infty.
\end{equation}
Furthermore,
\begin{equation}\label{eq:gradient-hessian-expansion}
|Du(x)|=O(|x|^{-1}),
\qquad
|D^2u(x)|=O(|x|^{-2})
\qquad \text{as } |x|\to +\infty.
\end{equation}
\end{theorem}


As a first consequence, the one-ended case is immediate.

\begin{corollary}\label{cor:one-end-plane}
Let $\Sigma\subset \mathbb{R}^3$ be a complete connected embedded $f$-surface satisfying \eqref{elliptic-conditions}, of finite total curvature, and with one end. Then $\Sigma$ is a plane.
\end{corollary}

\begin{proof}
By Theorem \ref{thm:log-expansion}, after a vertical translation the unique end is contained in one side of a horizontal plane outside a compact set. Since the remaining part of the surface is compact, after a further vertical translation the whole surface is contained in a half-space. Note that $\Sigma$ is properly immersed, since its end is a proper graph and the complement of the end in $\Sigma$ is compact. Hence Corollary \ref{thm:half-space} implies that $\Sigma$ is a plane.
\end{proof}

\section{A Jorge--Meeks type formula}\label{SectJM}

We now derive a Jorge--Meeks type formula for complete $f$-surfaces with finite total curvature and embedded ends.

\begin{theorem}\label{ThmJM}
Let $\Sigma\subset \r^3$ be a complete connected $f$-surface satisfying \eqref{elliptic-conditions}, of finite total curvature and with embedded ends. If $g$ denotes the genus of $\Sigma$ and $k$ its number of ends, then
\begin{equation}\label{JMFormula}
\int_\Sigma K\, dA = 4\pi(1-g-k).
\end{equation}
\end{theorem}

\begin{proof}
By Lemma \ref{LemFinite}, $\Sigma$ has finite conformal type, and each end is, after a rigid motion, the graph of a function defined over the complement of a compact disk in a plane. Moreover, by Theorem \ref{thm:log-expansion}, the gradient and Hessian of the graph tend uniformly to zero at infinity. More precisely, in the graph coordinates of an end, the induced metric reads $g=\langle\cdot,\cdot\rangle+du\otimes du$, with $|Du(x)|=O(|x|^{-1})$. Hence $g$ is uniformly equivalent to the Euclidean metric of the exterior planar domain, with constants tending to $1$ at infinity. In particular, the induced metric on each end is quasi-isometric to the Euclidean metric on an exterior planar domain; see \cite{Kanai}.

Since $K\leq 0$ and $\int_\Sigma |K|\, dA<+\infty$, we may apply the description of complete open surfaces of finite total curvature in \cite{Hulin-Troyanov}. In the terminology of \cite{Hulin-Troyanov}, each end carries a well-defined order, which is invariant under quasi-isometries and equals $-2$ for the flat metric on an exterior planar domain. The quasi-isometry above thus implies that each puncture corresponding to an embedded end has order $-2$ in the sense of \cite[Lemma 2.3]{Hulin-Troyanov}. Therefore, \cite[Theorem 2.9]{Hulin-Troyanov} yields
\[
\int_\Sigma K\, dA = 2\pi\Big(\chi(\overline{\Sigma})-2k\Big).
\]
Since $\chi(\overline{\Sigma})=2-2g$, we conclude that
\[
\int_\Sigma K\, dA = 2\pi(2-2g-2k)=4\pi(1-g-k),
\]
which proves \eqref{JMFormula}.
\end{proof}


\section{Alexandrov reflection method for catenoidal type ends}\label{SectAlexandrov}

A. Alexandrov \cite{Alexandrov1} characterized spheres as the only closed connected surfaces embedded in $\mathbb{R}^3$ with constant (non-zero) mean curvature. This is a major theorem, and the procedure introduced by Alexandrov has had a remarkable impact on the development of Differential Geometry. This technique is called the Alexandrov reflection method, and it is based on the maximum principle, one of the classical tools in the theory of second order elliptic partial differential equations; see \cite{Hopf3,Hopf4,Serrin}. 

\subsection{The Alexandrov reflection method}
\label{alexandrov-reflection}

Alexandrov's idea, roughly speaking, is to compare the surface with successive reflections of suitable portions of itself, looking for a first tangency point, and thereby obtaining a plane of symmetry for the original surface. In order to apply this method later, we present here the Alexandrov reflection method following \cite{AEG,Espinar,KKS}, which extends the classical compact argument to a non-compact setting.

Let $\Sigma$ be a connected properly embedded surface in $\mathbb{R}^3$. Then $\Sigma$ divides $\mathbb{R}^3$ into two connected components. Let us denote by $N$ a unit normal vector field globally defined on $\Sigma$. Set $\mathcal{C}(N)$ to be the component of $\mathbb{R}^3\setminus \Sigma$ pointed to by $N$. Note that $\partial\mathcal{C}(N)=\Sigma$.

Consider a fixed plane $\mathcal{P}\subset \mathbb{R}^3$ with normal unit vector $n$. For $t\in\mathbb{R}$ we denote by $\mathcal{P}_t$ the plane parallel to $\mathcal{P}$ at signed distance $t$, that is, $\mathcal{P}_t=\mathcal{P}+tn$. Hence, the family $\{\mathcal{P}_t\}_{t\in\mathbb{R}}$ is a foliation of $\mathbb{R}^3$ by planes parallel to $\mathcal{P}$. Set $\mathcal{P}_{t^-}$ and $\mathcal{P}_{t^+}$ to be the closed half-spaces, lower and upper respectively, determined by the plane $\mathcal{P}_t$, i.e.,
\begin{equation*}
\mathcal{P}_{t^-}=\underset{s\leq t}{\bigcup}\mathcal{P}_s
\qquad \mbox{and} \qquad
\mathcal{P}_{t^+}=\underset{s\geq t}{\bigcup}\mathcal{P}_s.
\end{equation*}
For any set $G\subset\mathbb{R}^3$, let $G_{t^+}$ be the portion of $G$ above the plane $\mathcal{P}_t$ and let $G_{t^+}^*$ be the reflection of this portion through the plane $\mathcal{P}_t$. Then we can write
\begin{equation}\label{reflection}
G_{t^+}=G\cap\mathcal{P}_{t^+}
\quad \mbox{and} \quad
G_{t^+}^*=\{p+(t-r)n:p\in\mathcal{P},\,p+(t+r)n\in G_{t^+}\};
\end{equation}
we define analogously the sets $G_{t^-}$ and $G_{t^-}^*$.

Given an open set $W\subset\mathcal{C}(N)$, consider the portion $S$ of the surface $\Sigma$ defined by $S=\partial W\cap \Sigma$. If the set $S_{t^+}$ is non-empty and $S_{t^+}^*\subset \overline{W}$, we write $S_{t^+}^*\geq S_{t^-}$, which means that the reflection of the set $S_{t^+}$ through the plane $\mathcal{P}_t$ lies above the set $S_{t^-}$, considering the height from the plane $\mathcal{P}$.

When $S$ is not empty, we define a function $\Lambda_1$, called the Alexandrov function associated with $S$.

We start with the definition of the domain of the function $\Lambda_1$, denoted by $\mathcal{D}$. Given a point $p\in\mathcal{P}$, let $L_p$ be the perpendicular line to $\mathcal{P}$ passing through $p$, so that
\[
L_p=\{p+tn:t\in\mathbb{R}\}.
\]
We say that a point $p\in\mathcal{P}$ belongs to $\mathcal{D}$ if there exists $t_1(p)\in\r$ such that
\[
\{p+tn:t> t_1(p)\}\cap\overline{W}=\emptyset,
\qquad
\mathbf{P}_1(p):=p+t_1(p)n\in S,
\]
and one of the following occurs (see Figure \ref{F6-art}):
\begin{itemize}
 \item[(i)] $L_p$ and $S$ are tangential at $\mathbf{P}_1(p)$;
 \item[(ii)] $L_p$ and $S$ are transversal at $\mathbf{P}_1(p)$ and there exists $t_2(p)\in\r$, defined as
 \[
 t_2(p):=\inf\{t\in\r : \{p+sn: t<s<t_1(p)\}\subset W\},
 \]
 such that
 \[
 \{p+tn:t_2(p)<t< t_1(p)\}\subset W
 \]
 and
 \[
 \mathbf{P}_2(p):=p+t_2(p)n\in S;
 \]
 \item[(iii)] $L_p$ and $S$ are transversal at $\mathbf{P}_1(p)$ and
 \[
 \{p+tn:t< t_1(p)\}\subset W.
 \]
\end{itemize}

\begin{figure}[ht!]
\centering
\includegraphics{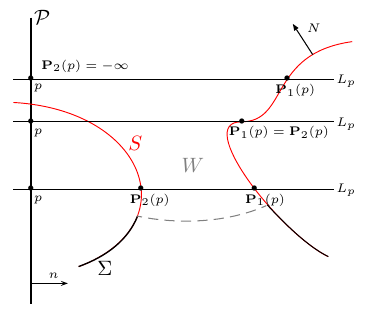}
\caption{$p \in \mathcal{D}$.}\label{F6-art}
\end{figure}

Geometrically, $\mathcal{D}$ contains all points $p$ in the plane $\mathcal{P}$ such that the lines $L_p$, coming from infinity, enter $\overline{W}$ through $S$ either tangentially, transversally while staying in $W$, or attempting to leave $\overline{W}$ again through $S$.

When $L_p$ and $S$ are tangential at $\mathbf{P}_1(p)$, we set
\[
\mathbf{P}_2(p)=\mathbf{P}_1(p),
\qquad
t_2(p)=t_1(p).
\]
We say that $\mathbf{P}_1(p)$ and $\mathbf{P}_2(p)$ are the first and second points of contact of the set $L_p\cap\overline{W}$ as $t$ decreases from $+\infty$, respectively.

Note that for a point $p\in \mathcal{D}$ it may happen that the line $L_p$ has a transversal first contact point $\mathbf{P}_1(p)\in S$ and a tangential second contact point $\mathbf{P}_2(p)\in S$ as $t$ decreases from $+\infty$, while still satisfying
\[
\{p+tn:t<t_2(p)\}\subset\overline{W}.
\]

\begin{remark}
Let $p\in\mathcal{P}$ be such that $L_p$ is transversal to $S = \Sigma \cap \partial W$ at $\mathbf{P}_1 (p)$. If there exists $t_2(p)\in\r$ such that
\[
\{p+tn:t_2(p)<t< t_1(p)\}\subset W
\]
and
\[
\mathbf{P}_2(p):=p+t_2(p)n\in \partial W \setminus S,
\]
then $p \not \in \mathcal{D}$ (see Figure \ref{F5-art}).

\begin{figure}[ht!]
\centering
\includegraphics{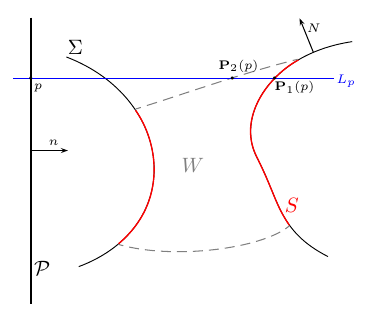}
\caption{$p \not \in \mathcal{D}$.}\label{F5-art}
\end{figure}

\end{remark}

We define $\Lambda_1:\mathcal{D}\rightarrow\{-\infty\}\cup\mathbb{R}$, the Alexandrov function associated with $S$, by
\begin{equation}\label{Alexandrov-function}
\Lambda_1(p)=\left\{
\begin{array}{cl}
\dfrac{t_1(p)+t_2(p)}{2}, & \mbox{if there exist }\mathbf{P}_1(p) \mbox{ and } \mathbf{P}_2(p), \mbox{ i.e. in cases (i) and (ii);}\\[3mm]
-\infty, & \mbox{if there exists only } \mathbf{P}_1(p), \mbox{ i.e. in case (iii).}
\end{array}
\right.
\end{equation}

When $\Lambda _1$ is finite, note that $\Lambda_1(p)$ is the value such that the reflection of the point $\mathbf{P}_1(p)$ through the plane $\mathcal{P}_{\Lambda_1(p)}$ is exactly $\mathbf{P}_2(p)$. Moreover, for every $t$ with $\Lambda_1(p)\leq t<t_1(p)$, the reflection of $\mathbf{P}_1(p)$ through the plane $\mathcal{P}_{t}$ is contained in $\overline{W}$. When $\Lambda_1(p)=-\infty$, the reflection of $\mathbf{P}_1(p)$ through the plane $\mathcal{P}_{t}$, for all $t < t_1(p)$, is contained in $W$.

\begin{remark}
For an illustration of the Alexandrov function associated with a surface and its domain, consider $\Sigma$ to be the complete vertical catenoid of necksize $1$, $\mathcal{C}(N)$ the connected component containing the revolution axis (the $z$-axis), $W$ the set of points in $\mathcal{C}(N)$ with non-negative third coordinate, and
\[
S=\partial W\cap \Sigma,
\]
that is, the upper vertical half-catenoid. If we take
\[
\mathcal{P}=\{z=0\}
\qquad\mbox{with normal } n=(0,0,-1),
\]
then
\[
\mathcal{D}=\{p\in\mathcal{P}:|p|\geq 1\}
\]
and
\begin{equation*}
\Lambda_1(p)=\left\{
\begin{array}{cl}
0, & \mbox{ if } |p|=1,\\[1mm]
-\infty, & \mbox{ if } |p|>1.
\end{array}
\right.
\end{equation*}
\end{remark}

Throughout this section, $\partial S$ denotes the boundary of $S$ regarded as a subset of $\partial W$, that is, the set of points of $S$ that are limits of points of $\partial W\setminus S$.

\begin{definition}
A point $p\in\mathcal{D}$ is called a local interior maximum for $\Lambda_1$ if there exists a neighborhood $U$ of $p$ in $\mathcal{P}$ such that for any $q\in U\cap \mathcal{D}$ we have
\[
\mathbf{P}_1(q),\mathbf{P}_2(q)\notin \partial S
\qquad\mbox{and}\qquad
\Lambda_1(q)\leq\Lambda_1(p).
\]
Any other local maximum of $\Lambda_1$ will be called a local boundary maximum.
\end{definition}

\begin{definition}\label{reflection-contact}
A first local point of reflection for $S$ with respect to the plane $\mathcal{P}$ with normal $n$ is defined to be a point $\mathbf{P}_2(p)$ such that $p\in\mathcal{D}$ and is a local maximum of $\Lambda_1$, that is, there exists a neighborhood $U$ of $p$ in $\mathcal{P}$ such that
\[
\Lambda_1(q)\leq\Lambda_1(p)
\qquad\mbox{for any } q\in U\cap\mathcal{D}.
\]
\end{definition}

The above definitions are justified through the next lemma, which shows that the Alexandrov reflection method can be applied to non-compact surfaces. We emphasize that the ambient surface $\Sigma$ has no boundary, whereas the subset $S=\partial W\cap\Sigma$ may have non-empty relative boundary $\partial S$ in $\Sigma$.

\begin{lemma}[Alexandrov reflection method]\label{Alexandrov-principle}
Let $\Sigma$ be a connected properly embedded $f$-surface. Let $\mathcal{P}$ be a plane in $\mathbb{R}^3$ with normal $n$. If, relative to the subsets $W\subset \mathcal{C}(N)$ and $S\subset \Sigma$, the Alexandrov function $\Lambda_1$ has a local interior maximum value $t_0$ at $p\in\mathcal{D}$, then the plane $\mathcal{P}_{t_0}$ is a plane of symmetry for $\Sigma$.
\end{lemma}

\begin{proof}
We compare the surface $S$ with the reflection $S_{t_0^+}^*$ of $S_{t_0^+}$ through the plane $\mathcal{P}_{t_0}$. Since $\Lambda_1(p)=t_0$, the point $\mathbf{P}_1(p)$ is reflected to $\mathbf{P}_2(p) \in S \cap S_{t_0^+}^*$.

By the definition of local interior maximum, there exists a neighborhood $U\subset\mathcal{P}$ of $p$ such that for every $q\in U\cap\mathcal{D}$ one has
\[
\mathbf{P}_1(q),\mathbf{P}_2(q)\notin \partial S
\qquad\mbox{and}\qquad
\Lambda_1(q)\leq t_0.
\]
Since $\mathbf{P}_1(p)\notin \partial S$, after shrinking $U$ if necessary we may choose $\epsilon>0$ such that $S\cap \mathbb{B}^3(\mathbf{P}_1(p),\epsilon)$ is a topological disk contained in the interior of $S$.

We claim that, after perhaps shrinking $U$ once more, every point $q\in U$ satisfies one of the following two alternatives:
\begin{itemize}
\item either $q\in\mathcal{D}$ and $\Lambda_1(q)\leq t_0$;
\item or $L_q\cap\overline{W}=\emptyset$ in a neighborhood of $\mathbf{P}_1(p)$.
\end{itemize}

Assume by contradiction that this is false. Then there exists a sequence $q_m\in U\setminus \mathcal{D}$ such that $L_{q_m}\cap\overline{W}\neq \emptyset$ near $\mathbf{P}_1(p)$. Let $\mathbf{Q}_1(q_m)$ denote the first contact point of $L_{q_m}$ with $\partial W$ coming from $+\infty$. Since $\mathbf{P}_1(p)$ lies in the interior of $S$, for all $m$ sufficiently large we have $\mathbf{Q}_1(q_m)\in S\cap \mathbb{B}^3(\mathbf{P}_1(p),\epsilon)$. Moreover, the line $L_{q_m}$ is transversal to $S$ at $\mathbf{Q}_1(q_m)$, because otherwise $q_m$ would belong to $\mathcal{D}$.

Because $\mathbf{Q}_1(q_m)\in \operatorname{int}(S)$ and $L_{q_m}$ is transversal to $S$
at $\mathbf{Q}_1(q_m)$, the line $L_{q_m}$ crosses $\partial W$ at that point.
Since $\mathbf{Q}_1(q_m)$ is the first contact point coming from $+\infty$, it follows
that there exists $\delta_m>0$ such that
\[
q_m+tn\in W
\qquad\text{for all } t_1(q_m)-\delta_m<t<t_1(q_m).
\]
In fact, $\delta_m$ does not degenerate: in case (ii) the line $L_p$ is transversal to the disk $S\cap\mathbb{B}^3(\mathbf{P}_1(p),\epsilon)$, so after shrinking $U$ all the lines $L_{q_m}$ cross this disk transversally and enter $W$ there, and they cannot meet $\overline{W}$ again before leaving the ball $\mathbb{B}^3(\mathbf{P}_1(p),\epsilon)$. Hence there is $\delta>0$, independent of $m$, with $\delta_m\geq\delta$, and the exit points $\widetilde{\mathbf{Q}}_m$ defined below stay at a definite distance below $\mathbf{P}_1(p)$. This covers the case in which $L_p$ is transversal to $S$ at $\mathbf{P}_1(p)$. If instead $L_p$ is tangential at $\mathbf{P}_1(p)$, so that $\mathbf{P}_2(p)=\mathbf{P}_1(p)\in\operatorname{int}(S)$, no uniform lower bound on $\delta_m$ is needed: the exit points $\widetilde{\mathbf{Q}}_m\in\partial W\setminus S$ constructed below then converge to $\mathbf{P}_1(p)$, and since $\operatorname{int}(S)$ is relatively open in $\partial W$, no point of $\operatorname{int}(S)$ can be a limit of points of $\partial W\setminus S$; this already yields the contradiction. Now $q_m\notin\mathcal{D}$ means precisely that the line $L_{q_m}$ has no second contact point with $S$ below $\mathbf{Q}_1(q_m)$. Therefore the first time at which the line stops belonging to $W$ must occur at a point of $\partial W\setminus S$.
If we define
\[
t_m:=\inf\Bigl\{t<t_1(q_m): q_m+\tau n\in W \text{ for all } t<\tau<t_1(q_m)\Bigr\},
\]
then
\[
\{q_m+tn:t_m<t<t_1(q_m)\}\subset W
\qquad\text{and}\qquad
\widetilde{\mathbf{Q}}_m:=q_m+t_m n\in \partial W\setminus S.
\]

Passing to a subsequence if necessary, the points $\mathbf{Q}_1(q_m)$ converge to $\mathbf{P}_1(p)$ and the corresponding line segments converge to $\{p+tn:t_2(p)<t<t_1(p)\}$. Hence $\widetilde{\mathbf{Q}}_m\longrightarrow \mathbf{P}_2(p)$. Since every $\widetilde{\mathbf{Q}}_m$ belongs to $\partial W\setminus S$, we conclude that $\mathbf{P}_2(p)\in \partial S$, which contradicts the fact that $p$ is a local interior maximum. This proves the claim.

Now let $q\in U$ be such that $\mathbf{P}_1(q)\in S_{t_0^+}$. By the claim we may assume that $q\in\mathcal{D}$, and then $\Lambda_1(q)\leq t_0$. By \eqref{Alexandrov-function}, this is equivalent to
\begin{equation}\label{inequality1}
t_2(q)\leq t_0-(t_1(q)-t_0).
\end{equation}
Looking at \eqref{reflection}, we see that \eqref{inequality1} means precisely that the reflection of $\mathbf{P}_1(q)$ through the plane $\mathcal{P}_{t_0}$ lies above the point $\mathbf{P}_2(q)$. Hence $\mathbf{P}_1(q)^*\in\overline{W}$. Summarizing, a neighborhood of $S_{t_0^+}^*$ containing $\mathbf{P}_2(p)$ is contained in $\overline{W}$. Besides, writing $N^*$ for the unit normal of the reflected surface $S_{t_0^+}^*$, the inclusion $\overline{W}\subset\overline{\mathcal{C}(N)}$ places $S_{t_0^+}^*$ on the side of $\Sigma$ into which $N$ points at $\mathbf{P}_2(p)$; hence $N^*(\mathbf{P}_1(p)^*)=N(\mathbf{P}_2(p))$.

Therefore, using Definition \ref{def:tangent-point}, $S_{t_0^+}^*\geq_{\mathbf{P}_2(p)} S$. In particular, $\mathbf{P}_2(p)$ is an interior tangent point when $\mathbf{P}_1(p)\neq \mathbf{P}_2(p)$ and a boundary tangent point when $\mathbf{P}_1(p)=\mathbf{P}_2(p)$. In the latter case $\mathbf{P}_2(p)=\mathbf{P}_1(p)$ lies on $\mathcal{P}_{t_0}$, and the interior-conormal condition of Definition \ref{def:tangent-point} holds: the reflection through $\mathcal{P}_{t_0}$ fixes the tangent plane $T_{\mathbf{P}_2(p)}\Sigma$ and maps the interior conormal of $\partial S$ at $\mathbf{P}_2(p)$ to that of $\partial S_{t_0^+}^*$, so the boundary form of the geometric maximum principle (Theorem \ref{geometric-maximum}) is the one applied. Let $R$ denote the reflection through $\mathcal{P}_{t_0}$. Both $R(\Sigma)$ and $\Sigma$ are complete $f$-surfaces without boundary, for the same $f$, and near $\mathbf{P}_2(p)$ the surface $R(\Sigma)$ contains the piece $\Sigma_{t_0^+}^*$, which lies on the $\overline{W}$ side of $\Sigma$. By Theorem \ref{geometric-maximum}, the surfaces $R(\Sigma)$ and $\Sigma$ coincide in a neighborhood of $\mathbf{P}_2(p)$, and then the continuation argument of Theorem \ref{geometric-maximum} gives $R(\Sigma)=\Sigma$; thus $\mathcal{P}_{t_0}$ is a plane of symmetry of $\Sigma$.
\end{proof}

The Alexandrov reflection method is very useful whenever we obtain a local interior maximum for the Alexandrov function. Although the function $\Lambda_1$ is not continuous in general, the next lemma shows that $\Lambda_1$ is upper semicontinuous with respect to planes as well as points. Combined with Lemma \ref{compactness} below, this will allow us to produce maximum points of the Alexandrov function on compact sets.

\begin{lemma}\label{tilted-planes}
Let $S\subset\Sigma$ be closed and let $\varepsilon\to0$. Suppose that there exists a sequence of points $p^\varepsilon\to p$ and a sequence of planes $\mathcal{P}^\varepsilon\to\mathcal{P}$ such that
\[
p^\varepsilon\in\mathcal{P}^\varepsilon
\qquad\mbox{and}\qquad
p\in\mathcal{P}.
\]
Let $\Lambda_1^\varepsilon$ and $\Lambda_1$ be the corresponding Alexandrov functions associated with $S$ and the planes $\mathcal{P}^\varepsilon$ and $\mathcal{P}$, respectively. If $\Lambda_1^\varepsilon(p^\varepsilon)$ exists for every $\varepsilon$, then either
\[
\limsup_{\varepsilon\to0} \Lambda_1^\varepsilon(p^\varepsilon)=-\infty,
\]
or $\Lambda_1(p)$ exists and
\begin{equation}\label{semicontinuous}
\limsup_{\varepsilon\to0} \Lambda_1^\varepsilon(p^\varepsilon)\leq\Lambda_1(p).
\end{equation}
\end{lemma}

\begin{proof}
Assume that $\Lambda_1^\varepsilon(p^\varepsilon)$ exists for $\varepsilon\to0$. Then there is a sequence $\left(\mathbf{P}_1(p^\varepsilon),\mathbf{P}_2(p^\varepsilon)\right)$ of pairs of points of $S$. If either $\mathbf{P}_2(p^\varepsilon)$, or $\mathbf{P}_1(p^\varepsilon)$ and hence $\mathbf{P}_2(p^\varepsilon)$, drifts off to infinity, then 
\[
\limsup_{\varepsilon\to0} \Lambda_1^\varepsilon(p^\varepsilon) = -\infty
\]
and the proof is finished. Thus we can assume that ${\rm dist}_{\mathbb{R}^3} \left( \mathbf{P}_i(p^\varepsilon ), \mathcal{P}^\varepsilon \right)$, $ i=1,2$, is bounded for all sufficiently small $\varepsilon$.

Since $S$ is closed, a subsequence converges to a pair of points of $S$, possibly identical, $(Q_1,Q_2)$. Clearly, $Q_1,Q_2\in L_p$, since $L_{p^\varepsilon}\to L_p$. We call $t_1$ and $t_2$ the heights of these points above $\mathcal{P}$.

We first observe that $p\in\mathcal{D}$, so that $\Lambda_1(p)$ exists. Indeed, the points of $L_p$ strictly above $Q_1$ are limits of points of the lines $L_{p^\varepsilon}$ lying strictly above $\mathbf{P}_1(p^\varepsilon)$, which do not belong to $\overline{W}$; hence the open ray of $L_p$ above $Q_1$ does not meet $W$, and the line $L_p$ first meets $\overline{W}$ at some point at height $t_1(p)\geq t_1$, which belongs to $S$ or produces, below it, the configuration of the definition of $\mathcal{D}$ through the segment inherited from the segments $\{p^\varepsilon+tn^\varepsilon : t_2(p^\varepsilon)<t<t_1(p^\varepsilon)\}\subset W$.

The values $t_1(p^\varepsilon)$ and $t_2(p^\varepsilon)$ converge to the heights of $Q_1$ and $Q_2$ above $\mathcal{P}$, respectively. By the definition of first and second contact points, the points $\mathbf{P}_1(p)$ and $\mathbf{P}_2(p)$ must be at least as high as $Q_1$ and $Q_2$, respectively, that is, $t_1(p)\geq t_1$ and $t_2(p)\geq t_2$. Hence $\Lambda_1(p)\geq\frac{t_1+t_2}{2}$, and \eqref{semicontinuous} follows.
\end{proof}


\subsection{The Alexandrov function and tilted planes}\label{subsect:tilted}

In this subsection, we prove that, given an unbounded end $E$ of an $f$-surface written as a graph with asymptotic expansion as in \eqref{eq:log-expansion}, and an orthogonal plane $\mathcal{P}$ with unit normal $n$, the maximum value of the Alexandrov function $\Lambda_1$ of $E$ with respect to the plane $\mathcal{P}$ is achieved at $\partial E$.

Without loss of generality, up to a rotation and a translation, we may assume that
\[
\partial E\subset\Pi := \{p\in \mathbb{R}^3 \, : \, x_3 =0\}
\]
and
\[
\lim _{|y_i| \to +\infty} N(y_i) = \nu = (0,0, \pm 1),
\]
for any proper sequence $y_i \in E$. Besides, after a reflection through the plane $\Pi$, if necessary, we may also assume that the graphical function $u$, given in Theorem \ref{thm:log-expansion}, is non-negative. Henceforth we assume that $\nu = (0,0,1)$, keeping in mind that the case $\nu = (0,0,-1)$ is completely analogous (see Figure \ref{F4-art}). 

\begin{figure}[ht!]
\centering
\includegraphics{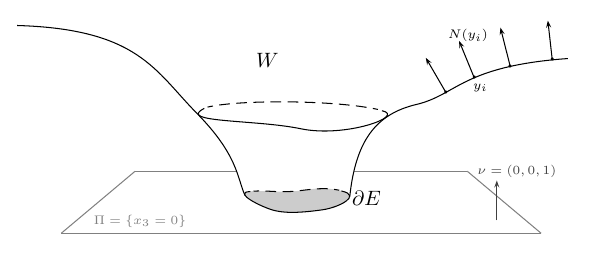}
\caption{$u \geq 0$ and $\lim _{p \in E, \, |p| \to \infty} N(p) =\nu=(0,0,1)$.}\label{F4-art}
\end{figure}

Let $W$ be the connected component of $\mathcal{C}(N)\cap\Pi^+$, where $\Pi^+:=\{p\in\mathbb{R}^3:x_3>0\}$ is the upper half-space, that contains on its boundary the bounded domain of $\Pi$ limited by $\partial E\cap\Pi$. Consider a vertical plane $\mathcal{P}$ with unit normal $n$ containing the vector $\nu=(0,0,1)$, and let
\[
h(p)=\langle p, (0,0,1)\rangle
\]
be the height function with respect to the plane $\Pi$. We define the Alexandrov function
\(
\Lambda:[0,\infty)\rightarrow \{-\infty\}\cup\mathbb{R}
\)
on $E$ (with respect to $\mathcal P$) by
\begin{equation}\label{associaded-Aleksandrov-function}
\Lambda(\rho)=\max\,\{\Lambda_1(p):p\in\mathcal{D},\,h(p)=\rho\},
\end{equation}
where $\mathcal{D}$ is the domain of the Alexandrov function $\Lambda_1$, given by \eqref{Alexandrov-function}, associated with $E$ with respect to $\mathcal P$ and the normal orientation given by $n$. Observe that the domain of the function $\Lambda_1$ is a subset of the plane $\mathcal{P}$, whereas the domain of $\Lambda$ is a subset of $\mathbb{R}$.

By Theorem \ref{thm:log-expansion}, for any $\rho\geq 0$, the level set
\(
E_\rho=\{q\in E:h(q)=\rho\}
\)
is a non-empty compact set. Then the function $\Lambda$ is finite-valued, since the set
\(
\{p\in\mathcal{D}:h(p)=\rho\}
\)
is non-empty and compact; indeed, the line through a point of $E_\rho$ extremal in the direction $n$ first meets $\overline{W}$ there, enters the bounded slice of $W$ at height $\rho$, and leaves it through another point of $E$, so its projection belongs to $\mathcal{D}$ and $\Lambda_1$ is finite at it. This situation changes drastically if we tilt the plane $\mathcal{P}$ slightly. We now study the consequences on the associated Alexandrov function $\Lambda_1$, and on the Alexandrov function $\Lambda$, caused by this tilting.

Let $n$ be the normal vector of the plane $\mathcal{P}$ and consider the vectors
\[
\nu^\varepsilon=\nu+\varepsilon n
\qquad\mbox{and}\qquad
n^\varepsilon=n-\varepsilon\nu
\]
for a small $\varepsilon>0$. Note that $\langle \nu^\varepsilon,n^\varepsilon\rangle=0$. We define the tilted planes
\[
\Pi_{\varepsilon}=\{p\in\mathbb{R}^3:\langle p,\nu^\varepsilon\rangle=0\}
\]
and $\mathcal{P}^\varepsilon$ to be the plane passing through the origin with normal vector $n^\varepsilon$. The height function from the plane $\Pi_{\varepsilon}$ is denoted by $h^\varepsilon$, namely
\[
h^\varepsilon(p)=\langle p,\nu^\varepsilon\rangle.
\]
Clearly,
\[
\mathcal{P}^\varepsilon\to\mathcal{P},
\qquad
\Pi_{\varepsilon}\to\Pi,
\qquad
h^\varepsilon\to h
\qquad\mbox{as } \varepsilon\to0.
\]

Since the end $E$ has at most logarithmic growth by Theorem \ref{thm:log-expansion}, an important phenomenon occurs when $\varepsilon>0$ is small: the Alexandrov function $\Lambda_1^\varepsilon$ associated with $E$ takes the value $-\infty$ outside a non-empty compact subset of $\mathcal{D}^\varepsilon$. More precisely:

\begin{lemma}\label{compactness}
Let $\varepsilon>0$ be sufficiently small. There exists a non-empty compact set $\Omega ^\varepsilon\subset\mathcal{P}^\varepsilon$ such that for any $p\in\mathcal{D}^\varepsilon \setminus \Omega ^\varepsilon$ there exists only a first contact point of the set
\[
\{p+tn^\varepsilon:t\in\mathbb{R}\}\cap\overline{W}
\]
as $t$ decreases from $+\infty$, i.e., $\Lambda_1^\varepsilon(p)=-\infty$. Moreover, there exists $p_0 \in \mathcal{D}^\varepsilon\cap\Omega^\varepsilon$ such that $\Lambda_1^\varepsilon(p_0)>-\infty$.
\end{lemma}

\begin{proof}
We prove this lemma in the case $n=(1,0,0)$. This is sufficient, since the asymptotic expansion in Theorem \ref{thm:log-expansion} is invariant under rotation of the $(x_1,x_2)$-coordinates. Set
\(
r=|(x_1,x_2)|,
\)
and write $N(x_1,x_2):=N(x_1,x_2,u(x_1,x_2))$ for the value of the unit normal at the corresponding point of the graph $E$. From the convergence $N(x_1,x_2)\to \nu$ as $r\to\infty$, we have
\[
\langle N(x_1,x_2),n^\varepsilon\rangle\to\langle \nu,n^\varepsilon\rangle=-\varepsilon<0,
\]
so we can choose $R_1>0$ large enough so that
\[
\langle N(x_1,x_2),n^\varepsilon\rangle\neq 0
\qquad\mbox{for all } r>R_1.
\]
Similarly, since $|Du(x_1,x_2)|\to 0$ as $r\to\infty$, there exists $R_2>0$, large enough, such that
\[
\frac{\partial u}{\partial x_1}>-\varepsilon
\qquad\mbox{for all } r>R_2.
\]
Let
\[
R=\max\{R_1,R_2\}
\qquad\mbox{and}\qquad
E_R=\{(x_1,x_2,x_3)\in E: x_1^2+x_2^2\leq R^2\}.
\]
Note that $E_R$ is a non-empty compact set, since $E_R$ is the graph of a continuous function over a compact set.

Consider the new coordinate system for $\mathbb{R}^3$ defined by
\[
(y_1,y_2,y_3)=(x_1-\varepsilon x_3,x_2,x_3+\varepsilon x_1).
\]
In these coordinates, the planes $\mathcal{P}^\varepsilon$ and $\Pi_{\varepsilon}$ are given by
\[
\mathcal{P}^\varepsilon=\{(y_1,y_2,y_3)\in\mathbb{R}^3:y_1=0\}
\qquad\mbox{and}\qquad
\Pi_{\varepsilon}=\{(y_1,y_2,y_3)\in\mathbb{R}^3:y_3=0\}.
\]
The projection of the set $E_R$ onto the plane $\mathcal{P}^\varepsilon$ is compact, so there exists $\rho>0$ such that this projection is contained in the set
\[
\Omega^\varepsilon=\{(0,y_2,y_3)\in\mathcal{P}^\varepsilon:|y_2|,|y_3|\leq \rho\}.
\]

The intersection of the end $E$ with the plane parallel to $\Pi_{\varepsilon}$ at signed height $\tau$ is given by the curve
\[
\Gamma_\tau=\left\{(y_1,y_2,\tau)\in\mathbb{R}^3:
u\left(\frac{y_1}{1+\varepsilon^2}+\frac{\varepsilon \tau}{1+\varepsilon^2},y_2\right)=
\frac{\tau}{1+\varepsilon^2}-\frac{\varepsilon y_1}{1+\varepsilon^2}\right\}.
\]
Let $\Phi$ be the function defined by
\[
\Phi(y_1,y_2)=u\left(\frac{y_1}{1+\varepsilon^2}+\frac{\varepsilon \tau}{1+\varepsilon^2},y_2\right)-
\frac{\tau}{1+\varepsilon^2}+\frac{\varepsilon y_1}{1+\varepsilon^2}.
\]
Hence,
\begin{equation}\label{derivate}
\frac{\partial \Phi}{\partial y_1}(y_1,y_2)=\frac{1}{1+\varepsilon^2}\frac{\partial u}{\partial x_1}\left(x_1,x_2\right)
+\frac{\varepsilon}{1+\varepsilon^2},
\end{equation}
where
\[
(x_1,x_2)=\left(\frac{y_1}{1+\varepsilon^2}+\frac{\varepsilon \tau}{1+\varepsilon^2},y_2\right).
\]

Let $p=(0,y_2,y_3)\in \mathcal{D}^\varepsilon \setminus \Omega ^\varepsilon$. Then the first contact point
\[
\mathbf{P}_1(p)=(y_1,y_2,y_3)=(x_1-\varepsilon x_3,x_2,x_3+\varepsilon x_1)
\]
of the line $L_p=\{p+tn^\varepsilon:t\in\r\}$ with $E$ must be transversal, since
\(
\langle N(x_1,x_2),n^\varepsilon\rangle\neq 0
\). We now show that there is no second contact point. There are two possibilities.

If $|y_3|=|\tau|>\rho$, then every point of $\Gamma_\tau$ lies outside $E_R$, because the projection of $E_R$ onto $\mathcal{P}^\varepsilon$ is contained in $\Omega^\varepsilon$. Therefore,
\[
\frac{\partial u}{\partial x_1}(x_1,x_2)>-\varepsilon
\]
at every point of $\Gamma_\tau$, and by \eqref{derivate} we obtain
\[
\frac{\partial \Phi}{\partial y_1}(y_1,y_2)>0
\]
along the whole curve $\Gamma_\tau$. Hence $\Gamma_\tau$ is the graph of a function of the variable $y_2$, and therefore each line
\(
\{(t,y_2,\tau)\in\mathbb{R}^3:t\in\mathbb{R}\}
\)
intersects $E$ exactly once. In particular, the line $L_p$ cannot have a second contact point.

If $|y_2|>\rho$, then every point of $E$ with second coordinate equal to $y_2$ lies outside $E_R$. Thus, along the whole line
\(
\{(t,y_2,\tau):t\in\mathbb{R}\}
\)
every intersection point with $E$ satisfies
\[
\frac{\partial u}{\partial x_1}(x_1,x_2)>-\varepsilon.
\]
Again by \eqref{derivate}, the function $\Phi(\,\cdot\,,y_2)$ is strictly increasing in $y_1$, and therefore that line intersects $E$ at most once. Hence $L_p$ has no second contact point.

In both cases we conclude that
\[
\Lambda_1^\varepsilon(p)=-\infty.
\]

We now prove that there exists $p_0 \in \mathcal{D}^\varepsilon\cap\Omega^\varepsilon$ such that $\Lambda_1^\varepsilon(p_0)>-\infty$. Consider the planes 
\[
\Pi_{\varepsilon}(s)=\{(y_1,y_2,y_3)\in\mathbb{R}^3:y_3= s\}
\qquad\mbox{for } s \in \r,
\]
and set
\[
s_0 := {\rm inf}\left\{ s \in \r : \Pi _\varepsilon (\bar s) \cap \partial E = \emptyset \mbox{ for all } \bar s > s\right\}.
\]
Clearly, $s_0 = {\rm max}_{\partial E} y_3> - \infty $, since $\partial E$ is compact. The plane $\Pi _\varepsilon (s_0)$ is the first plane ``coming from $+\infty$'' that touches $\partial E$.

Since $\partial E\subset \Pi=\{x_3=0\}$ and $E\setminus\partial E$ lies in the
side $\{x_3>0\}$, we can choose a compact collar neighborhood $A\subset E$ of
$\partial E$ whose boundary is $\partial A=\partial E\cup \Gamma$, where $\Gamma\subset E\setminus \partial E$ is a smooth embedded closed curve.
Moreover, shrinking the collar if necessary, we may assume that there exists
$\eta>0$ such that
\[
x_3\ge 3\eta \qquad \text{on } \Gamma .
\]

Since $A$ is compact and $h^\varepsilon\to h=x_3$ uniformly on $A$ as
$\varepsilon\to 0$, for $\varepsilon>0$ small enough we have
\[
h^\varepsilon \ge 2\eta \qquad \text{on } \Gamma .
\]
On the other hand, by definition, $s_0=\max_{\partial E} h^\varepsilon $. Hence, for $\varepsilon$ sufficiently small, we also have $s_0<\eta$. Fix now
any $s\in (s_0,\eta)$.

We claim that $\Pi_\varepsilon(s)\cap A$ contains a compact component. Indeed,
the function $h^\varepsilon|_A$ is continuous, satisfies $h^\varepsilon\le s_0 < s $ on $\partial E$, and $h^\varepsilon \ge 2\eta > s$ on $\Gamma $. Therefore the level set
\[
\Pi_\varepsilon(s)\cap A=\{q\in A:h^\varepsilon(q)=s\}
\]
is non-empty. Since it does not meet $\partial A$, it is a non-empty compact subset of the interior of $A$; in particular, $\Pi_\varepsilon(s)\cap E$ contains a compact component, say $C$.

Choose a point $q_0\in C$ where the $y_1$-coordinate attains its maximum or
minimum on $C$. Then the line
\[
L_{p_0}=\{(t,y_2(q_0),s):t\in\mathbb{R}\},
\qquad
p_0:=(0,y_2(q_0),s)\in \mathcal{P}^\varepsilon,
\]
is tangent to $C\subset E$ at $q_0$. Indeed, since $q_0$ is a $y_1$-extremum of $C=E\cap\{y_3=s\}$, we have $\partial\Phi/\partial y_1=0$ at $q_0$, which by \eqref{derivate} means $\partial u/\partial x_1=-\varepsilon$ there; as the unit normal of the graph is $N=(-Du,1)/\sqrt{1+|Du|^2}$, this gives $\langle N(q_0),n^\varepsilon\rangle=(-\partial u/\partial x_1-\varepsilon)/\sqrt{1+|Du|^2}=0$, so $n^\varepsilon$ lies in the tangent plane of $E$ at $q_0$. Thus $p_0\in \mathcal{D}^\varepsilon$, and by the tangential case in the definition of the Alexandrov function we obtain $\Lambda_1^\varepsilon(p_0)>-\infty$. Since $\Omega^\varepsilon$ was defined so as to contain all points where $\Lambda_1^\varepsilon\neq -\infty$, it follows that $p_0\in \mathcal{D}^\varepsilon\cap \Omega^\varepsilon $. This proves the existence of the required point $p_0$.
\end{proof} 

We now redefine the set $\Omega^\varepsilon$ as the closure of the set of points in $\mathcal{D}^\varepsilon$ where $\Lambda_1^\varepsilon$ is greater than $-\infty$, i.e.,
\[
\Omega^\varepsilon=\overline{\{p\in\mathcal{D}^\varepsilon:\Lambda_1^\varepsilon (p)\neq-\infty\}},
\]
which is still compact; see Figure \ref{boundaries}. Let us define $h_0^\varepsilon$ as the minimum of the height function $h^\varepsilon$ restricted to $\Omega^\varepsilon$, and the sets
\[
\mathcal{D}^\varepsilon_0=\{p\in\mathcal{D}^\varepsilon:h^\varepsilon(p)= h_0^\varepsilon\}
\qquad\mbox{and}\qquad
E^\varepsilon_0 = \{p\in E:h^\varepsilon(p)\geq h^\varepsilon_0\}.
\]

\begin{figure}[ht]
\centering
\includegraphics{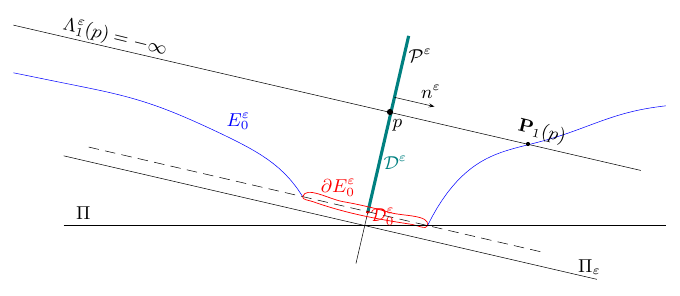}
\caption{The Alexandrov function $\Lambda_1^\varepsilon$ associated with $E$ is valued as $-\infty$ outside of a compact set.}
\label{boundaries}
\end{figure}

\begin{lemma}\label{decreasing}
For the end $E$, under all the assumptions made in this subsection, assuming moreover that $\Sigma$ has finitely many ends and that the limit unit normal of every end of $\Sigma$ is vertical, and for any vertical plane $\mathcal{P}$ containing $\nu=(0,0,1)$, either the function $\Lambda$ defined in \eqref{associaded-Aleksandrov-function} is strictly decreasing, or else $\Sigma$ has a plane of reflection parallel to $\mathcal{P}$.
\end{lemma}

\begin{proof}
We begin by proving that the function $\Lambda$ is non-increasing. For this it suffices to show that 
\[
\Lambda(\rho)\leq\Lambda(0)\qquad\mbox{for all }\rho>0,
\]
since we can translate vertically $E$ and redefine the end so as to choose the level $\rho=0$ arbitrarily.

Recall that, as in the setup of this subsection, 
\[
W\ \text{is the component of\ }\mathcal{C}(N)\cap\{p\in\mathbb{R}^3:h(p)>0\}
\]containing on its boundary the bounded domain of $\Pi$ limited by $\partial E$, where $N$ is the unit normal vector field defined on $E$ converging to $\nu$ at infinity. We claim:

\setcounter{claim}{0}
\begin{claim}\label{claim1}
$\Lambda(\rho)\leq\Lambda(0)$ for all $\rho>0$ if, and only if, 
\[
E_{t^+}^*\cap\{p\in\mathbb{R}^3:h(p)>0\}\subset \overline{W}
\qquad\mbox{for all } t>\Lambda(0).
\]
\end{claim}

\begin{proof}[Proof of Claim A]
Here $E_{t^+}^*$ is the reflection of $E_{t^+}$ through a plane parallel to $\mathcal{P}$ at distance $t$, following the notation of Section \ref{alexandrov-reflection}. 
If $\Lambda(\rho)\leq\Lambda(0)$ for all $\rho>0$, then for an arbitrary $t>\Lambda(0)$ we have
\[
\Lambda(\rho)\leq t
\qquad\mbox{for all }\rho>0.
\]
By the definition of the Alexandrov function, we infer that
\[
E_{t^+}^*\cap\{p\in\mathbb{R}^3:h(p)>0\}\subset \overline{W}.
\]

Conversely, suppose that
\[
E_{t^+}^*\cap\{p\in\mathbb{R}^3:h(p)>0\}\subset \overline{W}
\qquad\mbox{for all } t>\Lambda(0).
\]
We argue by contradiction. Assume that there exist $\rho_0>0$ and $p\in\mathcal{D}$ such that
\[
h(p)=\rho_0
\qquad\mbox{and}\qquad
\Lambda_1(p)=\Lambda(\rho_0)>\Lambda(0).
\]
Set $t=\Lambda(\rho_0)$. If a neighborhood of the reflected set $E_{t^+}^*$ containing the reflection of $\mathbf{P}_1(p)$ through the plane $\mathcal{P}_t$ were contained in $\overline{W}$, then the reflection of $\mathbf{P}_1(p)$ through $\mathcal{P}_t$, namely $\mathbf{P}_2(p)$, would be a tangent point between $E_{t^+}^*$ and $E$ with $E_{t^+}^*\geq_{\mathbf{P}_2(p)}E$, exactly as in the proof of Lemma \ref{Alexandrov-principle}, and the geometric maximum principle would imply that $\mathcal{P}_t$ is a plane of symmetry of $E$. But then every horizontal section of $E$ would be symmetric with respect to $\mathcal{P}_t$, and therefore
\[
\Lambda(\rho)=t
\qquad\mbox{for all }\rho\geq 0.
\]
In particular, $\Lambda(0)=t$, contradicting $t>\Lambda(0)$. Hence no such neighborhood can be contained in $\overline{W}$, and Claim \ref{claim1} follows.
\end{proof}

To prove that
\[
E_{t^+}^*\cap\{p\in\mathbb{R}^3:h(p)>0\}\subset \overline{W}
\qquad\mbox{for all } t>\Lambda(0),
\]
we use tilted planes together with Lemma \ref{tilted-planes}. 

The semicontinuity of the Alexandrov function associated with $E$, Lemma \ref{tilted-planes}, and Lemma \ref{compactness} imply that the maximum value of the function $\Lambda_1^\varepsilon$ must be attained at some point of the compact set $\mathcal{D}^\varepsilon\cap \Omega^\varepsilon$. Indeed, by Lemma \ref{compactness} the supremum of $\Lambda_1^\varepsilon$ is not $-\infty$ and any maximizing sequence stays in $\Omega^\varepsilon$; a subsequence converges to some $q\in\Omega^\varepsilon$, and Lemma \ref{tilted-planes}, applied with the constant plane $\mathcal{P}^\varepsilon$, shows that $q\in\mathcal{D}^\varepsilon$ and that $\Lambda_1^\varepsilon(q)$ bounds the supremum from above, so the supremum is attained at $q$. We now prove that the height of such a point, measured from the plane $\Pi_{\varepsilon}$, is exactly $h_0^\varepsilon$.

\begin{claim}\label{claim2}
The function $\Lambda_1^\varepsilon$ attains its maximum at some point of $\mathcal{D}^\varepsilon_0\cap \Omega^\varepsilon$.
\end{claim}

\begin{proof}[Proof of Claim B]
By the previous discussion, the function $\Lambda_1^\varepsilon$ attains its
maximum value $t$ at some point $q\in\mathcal{D}^\varepsilon\cap \Omega^\varepsilon$. We claim that necessarily
\[
h^\varepsilon(q)=h_0^\varepsilon.
\]

Assume by contradiction that $h^\varepsilon(q)>h_0^\varepsilon$. Since both contact points $\mathbf{P}_1(q)$ and $\mathbf{P}_2(q)$ lie on the line
$L_q$, and $h^\varepsilon$ is constant along $L_q$, we have
\[
h^\varepsilon(\mathbf{P}_1(q))=h^\varepsilon(\mathbf{P}_2(q))=h^\varepsilon(q)>h_0^\varepsilon.
\]
Hence $\mathbf{P}_1(q),\mathbf{P}_2(q)\in \operatorname{int}(E_0^\varepsilon)$, so both points lie away from $\partial E_0^\varepsilon$. Consider then the open set $W^\varepsilon\subset W$ of points of $W$ lying strictly above the plane $\Pi_\varepsilon(h_0^\varepsilon)$, so that $S^\varepsilon:=\partial W^\varepsilon\cap E\supset \operatorname{int}(E_0^\varepsilon)$, and the Alexandrov function of $S^\varepsilon$ agrees with $\Lambda_1^\varepsilon$ near $q$. Since $\mathbf{P}_1(q)$ and $\mathbf{P}_2(q)$ avoid $\partial S^\varepsilon$, and the contact points of nearby lines stay close to those of $L_q$, by the continuity argument in the proof of Lemma \ref{Alexandrov-principle}, the point $q$ is a local interior maximum of this Alexandrov function (see
Figure \ref{claimB}). By Lemma \ref{Alexandrov-principle}, applied to the surface $\Sigma$ with the subsets $W^\varepsilon$ and $S^\varepsilon$, the surface $\Sigma$ would then have a plane of symmetry parallel to $\mathcal{P}^\varepsilon$. This is impossible: by hypothesis every end of $\Sigma$ has vertical limit Gauss map, while a reflection in a plane parallel to $\mathcal{P}^\varepsilon$, which is non-vertical for $\varepsilon>0$, would send these vertical limit normals to non-vertical ones. This contradiction proves that $ h^\varepsilon(q)=h_0^\varepsilon$, that is, $q\in \mathcal{D}^\varepsilon_0\cap \Omega^\varepsilon$. Hence Claim \ref{claim2} follows.
\end{proof}

\begin{figure}[ht!]
\centering
\includegraphics{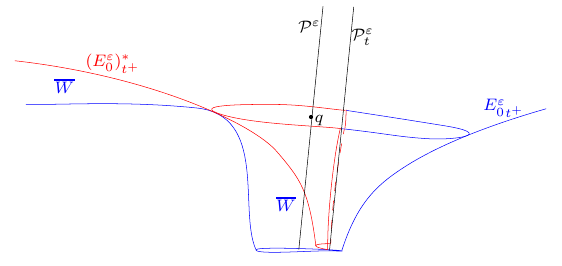}
\caption{$\Lambda_1^\varepsilon$ attains its maximum value $t$ at $q\in\mathcal{D}^\varepsilon\cap \Omega^\varepsilon$ such that $h^\varepsilon(q)>h_0^\varepsilon$.}
\label{claimB}
\end{figure}

Writing $z^\varepsilon$ for the maximum value of $\Lambda_1^\varepsilon$, it follows that
\begin{equation}\label{contained-1}
(E^\varepsilon_0)_{t^+}^*\subset \overline{W}
\qquad\mbox{for all }t\geq z^\varepsilon.
\end{equation}
Letting $\varepsilon\to 0$ in \eqref{contained-1} yields
\begin{equation}\label{contained-2}
E_{t^+}^*\subset \overline{W}
\qquad\mbox{for all }t\geq \limsup_{\varepsilon\to0}z^\varepsilon.
\end{equation}
Indeed, fix $t>\limsup_{\varepsilon\to0}z^\varepsilon$ and let $q\in E_{t^+}^*$ be the reflection through $\mathcal{P}_{t}$ of a point $r\in E$ with $h(r)>t$. For $\varepsilon$ small enough we have $t\geq z^\varepsilon$ and, since $h^\varepsilon\to h$ uniformly on compact subsets of $E$ and $h_0^\varepsilon\to 0$, the point $r$ belongs to $E_0^\varepsilon$; its reflection $q_\varepsilon$ through $\mathcal{P}^\varepsilon_{t}$ then lies in $(E_0^\varepsilon)_{t^+}^*\subset\overline{W}$ by \eqref{contained-1}, and $q_\varepsilon\to q$ as $\varepsilon\to0$. As $\overline{W}$ is closed, $q\in\overline{W}$.
We now check that $\limsup_{\varepsilon\to0}z^\varepsilon\leq \Lambda(0)$. By Claim \ref{claim2}, the value $z^\varepsilon$ is attained at a point $q_\varepsilon\in\mathcal{D}^\varepsilon_0\cap\Omega^\varepsilon$, of height $h^\varepsilon(q_\varepsilon)=h_0^\varepsilon$. The heights $h_0^\varepsilon$ tend to $0$ as $\varepsilon\to 0$: on the one hand, the construction of the point $p_0$ in the proof of Lemma \ref{compactness} produces points of $\Omega^\varepsilon$ at heights arbitrarily close to $\max_{\partial E}h^\varepsilon\to 0$, so $\limsup h_0^\varepsilon\leq 0$; on the other hand, since $h^\varepsilon$ is constant along each line $L_p$ (because $\langle\nu^\varepsilon,n^\varepsilon\rangle=0$), every $p\in\mathcal{D}^\varepsilon$ satisfies $h^\varepsilon(p)=h^\varepsilon(\mathbf{P}_1(p))\geq\inf_{\overline{W}}h^\varepsilon\to 0$, so $\liminf h_0^\varepsilon\geq 0$. If $\limsup z^\varepsilon=-\infty$ there is nothing to prove. Otherwise, along a subsequence realizing the $\limsup$, the points $q_\varepsilon$ converge to some $q\in\mathcal{P}$ with $h(q)=0$, and Lemma \ref{tilted-planes} shows that $q\in\mathcal{D}$ and
\[
\limsup_{\varepsilon\to0}z^\varepsilon=\limsup_{\varepsilon\to0}\Lambda_1^\varepsilon(q_\varepsilon)\leq \Lambda_1(q)\leq \Lambda(0),
\]
where the last inequality uses the definition \eqref{associaded-Aleksandrov-function} of $\Lambda$. Given $t\geq\Lambda(0)$, we have
\[
t\geq \limsup_{\varepsilon\to0}z^\varepsilon,
\]
hence \eqref{contained-2} yields $E_{t^+}^*\cap\{p\in\mathbb{R}^3:h(p)>0\}\subset \overline{W}$. Using Claim \ref{claim1}, we conclude that the function $\Lambda$ is non-increasing.

If $\Lambda$ is not strictly decreasing, then there exist $0\leq \rho_1<\rho_2$ such that
\[
\Lambda(\rho_1)=\Lambda(\rho_2).
\]
Since $\Lambda$ is non-increasing, it follows that $\Lambda$ is constant on $[\rho_1,\rho_2]$. In particular, the Alexandrov function associated with $E$, namely $\Lambda_1$, has a local interior maximum, and Lemma \ref{Alexandrov-principle} yields a plane of symmetry parallel to $\mathcal{P}$. This proves the lemma.
\end{proof}


\subsection{A Schoen type theorem for $f$-surfaces}

In this subsection we establish a Schoen-type theorem for $f$-surfaces.

\begin{lemma}\label{lem:parallel-ends}
Let $\Sigma\subset\mathbb{R}^3$ be a complete connected embedded $f$-surface satisfying \eqref{elliptic-conditions}, of finite total curvature, and with two embedded ends. Then the two ends are parallel and $\Sigma$ diverges to infinity in opposite directions.

More precisely, after a rigid motion there exist a plane
\(
\Pi=\{x_3=0\},
\)
a compact set $K\subset \Sigma$, and smooth functions $u_+,u_-:\mathbb{R}^2\setminus D(R)\to \mathbb{R}$ such that $\Sigma\setminus K=E_+\cup E_-$, where
\[
E_+=\{(x,u_+(x)):x\in\mathbb{R}^2\setminus D(R)\},
\qquad
E_-=\{(x,u_-(x)):x\in\mathbb{R}^2\setminus D(R)\},
\]
with
\[
u_+(x)>0,
\qquad
u_-(x)<0,
\]
and
\[
u_+(x)\to +\infty,
\qquad
u_-(x)\to -\infty
\qquad\mbox{as } |x|\to +\infty.
\]
Moreover, the unit normal of $\Sigma$ tends to opposite vertical limits along the two ends.
\end{lemma}

\begin{proof}
By Lemma \ref{LemFinite}, each end is a graph over the complement of a compact disk in some plane, and the Gauss map extends continuously to the punctures. Let $\nu_1$ and $\nu_2$ be the limiting unit normals at the two punctures.

If $\nu_1$ and $\nu_2$ were not parallel, then the two graphical ends would therefore intersect outside a compact set, contradicting embeddedness. Indeed, let $\Pi_1$ and $\Pi_2$ be the two limit planes; being non-parallel, they meet along an affine line $\ell$. Choose a unit vector $w$ transverse to both $\Pi_1$ and $\Pi_2$, let $\Pi_0$ be a plane orthogonal to $w$, and let $\pi:\mathbb{R}^3\to\Pi_0$ be the projection along $w$. By Theorem \ref{thm:log-expansion}, for $|x|\geq R_0$ each end $E_i$ is a graph over $\Pi_i\setminus D(R_0)$ with height $O(\log|x|)$ and gradient $O(|x|^{-1})$. Since $w$ is transverse to $\Pi_i$ and the gradient tends to $0$, for $R_0$ large the tangent planes of $E_i$ stay uniformly transverse to $w$. Moreover, $\pi$ restricts to an affine isomorphism $\Pi_i\to\Pi_0$, and the graph map $x\mapsto(x,u_i(x))$ is a $C^1$ perturbation of it whose perturbation tends to $0$ at infinity; hence, for $R_0$ large, $\pi|_{E_i}$ restricts to a proper local diffeomorphism from $E_i\cap\{|x|\geq R_0\}$ onto the exterior of a compact set of $\Pi_0$. Being a proper local diffeomorphism between connected oriented surfaces, it is a covering map with constant, orientation-preserving sheet number (the Jacobian of $\pi|_{E_i}$ has constant sign); the proper homotopy $x\mapsto\pi(x,t\,u_i(x))$, $t\in[0,1]$, to the affine isomorphism $\pi|_{\Pi_i}$ shows this number is $1$. Hence $\pi|_{E_i}$ is a proper diffeomorphism onto the exterior of a compact set of $\Pi_0$.

Over each $z\in\Pi_0$ with $|z|$ large there is then a unique point of $E_i$, whose signed $w$-coordinate we denote by $g_i(z)$. Writing $a_i(z)$ for the $w$-coordinate of the point of the \emph{plane} $\Pi_i$ over $z$, and noting that $\pi$ restricts to a linear isomorphism $\Pi_i\to\Pi_0$, so that $|x|$ and $|z|$ are comparable, the $O(\log|x|)$ deviation of the graph from $\Pi_i$ yields $g_i(z)=a_i(z)+O(\log|z|)$, where each $a_i$ is an affine function of $z$. As $\Pi_1$ and $\Pi_2$ are non-parallel, $\delta:=a_1-a_2$ is a non-constant affine function and $\delta$ changes sign. Therefore $h(z):=g_1(z)-g_2(z)=\delta(z)+O(\log|z|)$ tends to $+\infty$ in one direction and to $-\infty$ in the opposite one. Joining two such points by an arc contained in the connected set $\{|z|=\rho\}$ for $\rho$ large and applying the intermediate value theorem, we find $z^\ast$ with $h(z^\ast)=0$; the corresponding points of $E_1$ and $E_2$ share the same projection $z^\ast$ and the same $w$-coordinate, and hence coincide. This produces a point of $E_1\cap E_2$, contradicting the embeddedness of $\Sigma$. Hence the ends are parallel. After a rigid motion we may assume that the common limiting plane is $\Pi=\{x_3=0\}$.

Again by Theorem \ref{thm:log-expansion}, the height function of each end admits a limit in $\mathbb{R}\cup\{\pm\infty\}$. We claim that the two limits are infinite and of opposite sign. Suppose, on the contrary, that one of the limits is finite, or that both are infinite of the same sign. In each of these cases, both ends are contained, outside a compact set, in a half-space determined by a horizontal plane: this is clear if both limits are finite or both are equal to $+\infty$ (resp. $-\infty$), and if one limit is finite and the other is $+\infty$ (resp. $-\infty$) the surface again lies above (resp. below) some horizontal plane outside a compact set. Since the remaining part of $\Sigma$ is compact, the whole surface would then be contained in a half-space. The arguments of Corollary \ref{thm:half-space} would imply that $\Sigma$ is a plane, which is impossible since $\Sigma$ has two ends. This proves the claim.

Denote by $E_+$ the end whose height tends to $+\infty$ and by $E_-$ the one whose height tends to $-\infty$, and write $u_\pm$ for the corresponding graphing functions. Enlarging the compact set $K$ if necessary, we get $u_+>0$ and $u_-<0$, together with $u_\pm(x)\to\pm\infty$ as $|x|\to+\infty$.

Finally, we check that the unit normal $N$ of $\Sigma$ has opposite vertical limits along the two ends. Since $\Sigma$ is connected, properly embedded, and two-sided (it carries the globally defined unit normal $N$), it separates $\mathbb{R}^3$ into exactly two connected components, each having $\Sigma$ as its boundary, as in Section \ref{alexandrov-reflection}; let $\mathcal{C}(N)$ be the one into which $N$ points. Fix $x$ with $|x|$ large and let $L$ be the upward-oriented vertical line over $x$. Since $\Sigma\setminus K=E_+\cup E_-$ with $K$ compact, for $x$ outside the projection of $K$ the line $L$ meets $\Sigma$ in exactly two points, once on $E_-$ and, higher up, once on $E_+$; both intersections are transversal, because $|Du_\pm|=O(|x|^{-1})$ forces the tangent planes of the ends to be nearly horizontal. These two points split $L$ into the ray $\alpha$ below $E_-$, the segment $\beta$ between the two ends, and the ray $\gamma$ above $E_+$. Each transversal crossing passes from one component of $\mathbb{R}^3\setminus\Sigma$ to the other; hence, after the two crossings, $L$ returns to the component it started in, so exactly one of the two crossings enters $\mathcal{C}(N)$ while the other leaves it. As $N$ points into $\mathcal{C}(N)$ and $L$ is oriented upward, $\langle N,e_3\rangle>0$ at the entering crossing and $\langle N,e_3\rangle<0$ at the leaving one. Letting $|x|\to+\infty$, the nearly horizontal normals give $\langle N,e_3\rangle\to\pm1$, with opposite signs on $E_+$ and $E_-$; therefore $N$ tends to opposite vertical limits.
\end{proof}

\begin{theorem}\label{main-theorem}
Let $\Sigma\subset\mathbb{R}^3$ be a complete connected embedded $f$-surface satisfying \eqref{elliptic-conditions}, of finite total curvature, and with two embedded ends. Then $\Sigma$ must be rotationally symmetric. In particular, there exists $\tau>0$ such that $\Sigma$ is the surface of revolution $M_\tau$ determined by Sa Earp and Toubiana in \cite{SaEarp1}.
\end{theorem}

\begin{proof}
Let $E_1$ and $E_2$ be the ends of $\Sigma$. By Lemma \ref{lem:parallel-ends} we know that $E_1$ and $E_2$ are parallel ends and that $\Sigma$ diverges to infinity in opposite directions. 

Let $\nu$ be the continuous extension of the Gauss map of one of the ends; observe that the Gauss map tends to $-\nu$ at the other end. Let $\mathcal{P}$ be a plane containing $\nu$. Up to a rigid motion we may assume that $\nu=(0,0,1)$ and that $\mathcal P$ is a vertical plane. Consider the Alexandrov functions on each end $E_1$ and $E_2$ (with respect to $\mathcal P$), denoted by $\Lambda^1$ and $\Lambda^2$ respectively, and defined in \eqref{associaded-Aleksandrov-function}. Then, by Lemma \ref{decreasing}, either both of these Alexandrov functions are strictly decreasing, or at least one of the ends has a plane of symmetry parallel to $\mathcal{P}$.

In the second case, if $\mathcal{Q}$ is a plane of symmetry of one of the ends parallel to $\mathcal{P}$, then the reflection of $\Sigma$ through $\mathcal{Q}$ is an $f$-surface that coincides with $\Sigma$ along that end; since they agree on an open set they have, in particular, a tangent point with the ordering required by Theorem \ref{geometric-maximum}, so by the geometric maximum principle, Theorem \ref{geometric-maximum}, the two surfaces coincide everywhere, and $\Sigma$ has a plane of symmetry parallel to $\mathcal{P}$.

Assume now that both Alexandrov functions are strictly decreasing. In analogy with \eqref{associaded-Aleksandrov-function}, define
\[
\Lambda_\Sigma(\rho)=\max\{\Lambda_1(p):p\in\mathcal{D},\ h(p)=\rho\},
\qquad \rho\in\mathbb{R},
\]
where $h$ is the height function with respect to $\Pi$. Since the two ends diverge to infinity in opposite directions, every horizontal slice $\Sigma\cap\{h=\rho\}$ is compact and non-empty, so $\mathcal{D}\cap\{h=\rho\}$ is compact; moreover, the horizontal line through a point of the slice $\Sigma\cap\{h=\rho\}$ extremal in the direction $n$ first meets $\overline{W}$ there, enters the bounded slice of $W$, and must leave it through another point of $\Sigma$, so the projection of that point belongs to $\mathcal{D}$ and $\Lambda_1$ is finite at it. Arguing as in Section \ref{subsect:tilted}, the maximum above is attained and $\Lambda_\Sigma(\rho)\in\mathbb{R}$ for every $\rho$.

There exist $a<b$ such that the part of $\Sigma$ above $\{h=b\}$ lies in $E_+$ and the part below $\{h=a\}$ lies in $E_-$. For $\rho\geq b$, a horizontal line at height $\rho$ meets $\overline{W}$ in the configuration described in Section \ref{subsect:tilted} for the end $E_+$, so $\Lambda_\Sigma$ agrees on $[b,+\infty)$ with the Alexandrov function of $E_+$, which is strictly decreasing. Likewise, applying Section \ref{subsect:tilted} to $E_-$ after the reflection $x_3\mapsto -x_3$, the function $\Lambda_\Sigma$ is strictly increasing on $(-\infty,a]$. In particular,
\[
M:=\sup_{\rho\in\mathbb{R}}\Lambda_\Sigma(\rho)=\sup_{\rho\in[a,b]}\Lambda_\Sigma(\rho)<+\infty .
\]
Let $p_j\in\mathcal{D}$, $a\leq h(p_j)\leq b$, be a maximizing sequence, $\Lambda_1(p_j)\to M$. The contact points $\mathbf{P}_1(p_j)$ lie on the compact set $\Sigma\cap\{a\leq h\leq b\}$, so the points $p_j$ stay in a compact subset of $\mathcal{P}$ and a subsequence converges to some $p\in\mathcal{P}$. Since $M>-\infty$, Lemma \ref{tilted-planes}, applied with the constant sequence of planes $\mathcal{P}^\varepsilon=\mathcal{P}$, gives $p\in\mathcal{D}$ and $\Lambda_1(p)\geq M$, whence $\Lambda_1(p)=M$ and the supremum is attained. As $S=\Sigma$ has empty boundary, $p$ is a local interior maximum of $\Lambda_1$, and Lemma \ref{Alexandrov-principle} yields a plane of symmetry parallel to $\mathcal P$ (see Figure \ref{Alexandrov-graph}).

\begin{figure}[h]
\centering
\includegraphics{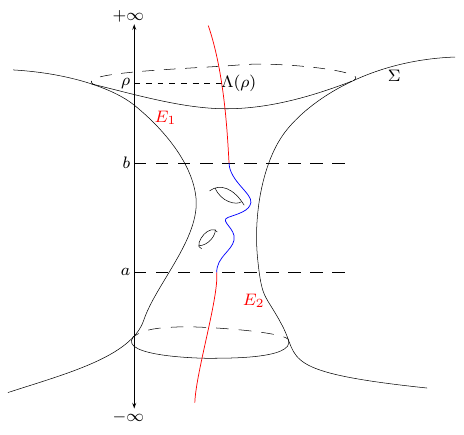}
\caption{The Alexandrov function, $\Lambda$, on $\Sigma$.}
\label{Alexandrov-graph}
\end{figure}

Therefore, for any vertical plane $\mathcal{P}$ containing $\nu$, the surface $\Sigma$ is symmetric with respect to some plane parallel to $\mathcal{P}$. Choose a regular value $t$ of the height function $x_{3|\Sigma}$. Since $\Sigma$
has two ends and diverges to infinity in opposite directions, the horizontal section
\[
\Sigma_t:=\Sigma\cap\{x_3=t\}
\]
is a compact $1$-dimensional submanifold of $\mathbb{R}^3$, hence a finite union
of smooth embedded closed curves. Its barycenter is therefore well defined.
Every vertical plane of symmetry of $\Sigma$ leaves $\Sigma_t$ invariant, and
consequently it fixes the barycenter of $\Sigma_t$. Hence all vertical planes of
symmetry pass through the same point of the horizontal plane $\{x_3=t\}$.
Since each of them is parallel to $\nu$, they intersect along a common line parallel
to $\nu$. Therefore $\Sigma$ is rotationally symmetric with respect to this line. By the converse statement in Theorem \ref{existence}, every such surface is one of the examples $M_\tau$; hence
\[
\Sigma=M_\tau,
\]
where $\tau>0$ is the distance from $\Sigma$ to its axis of revolution.
\end{proof}

\begin{remark}
By Theorem \ref{existence}, the surface has a horizontal plane of symmetry. In particular, the logarithmic coefficients of the two ends in the expansion \eqref{eq:log-expansion} coincide.
\end{remark}


\section{Classification for small total curvature}\label{SectApJM}

We conclude with the classification of complete connected embedded $f$-surfaces of finite total curvature and small absolute total curvature.

\begin{theorem}\label{thm:small-total-curvature}
Let $\Sigma \subset \r^3$ be a complete connected embedded $f$-surface satisfying \eqref{elliptic-conditions} and having finite total curvature. Then:
\begin{itemize}
\item if $\int _\Sigma |K| \, dA < 4 \pi$, $\Sigma$ is a plane;
\item if $\int _\Sigma |K| \, dA < 8 \pi$, $\Sigma$ is either a plane or a special catenoid.
\end{itemize}
\end{theorem}

\begin{proof}
By Theorem \ref{ThmJM},
\[
\int_\Sigma K\, dA = 4\pi(1-g-k),
\]
where $g$ is the genus of $\Sigma$ and $k$ is the number of ends. Since $K\leq 0$, by property (i) in Section \ref{SectPre}, we have $\int_\Sigma |K|\,dA=-\int_\Sigma K\,dA=4\pi(g+k-1)$. Since $\Sigma$ is complete, connected, embedded, and there are no closed $f$-surfaces in $\r^3$, we have $k\geq 1$.

Assume first that $\int_\Sigma |K|\, dA<4\pi $. Then $ g+k-1<1$, hence $g+k<2$. Since $g\geq 0$ and $k\geq 1$ are integers, it follows that $g=0$ and $ k=1$. Therefore $\Sigma$ has one embedded end, and Corollary \ref{cor:one-end-plane} implies that $\Sigma$ is a plane.

Assume now that $\int_\Sigma |K|\, dA<8\pi $. Then $ g+k-1<2$, hence $g+k<3$. Since $g\geq 0$ and $k\geq 1$ are integers, the only possibilities are
\[
(g,k)=(0,1),\qquad (g,k)=(1,1),\qquad (g,k)=(0,2).
\]

If $(g,k)=(0,1)$, then $\Sigma$ is a plane by Corollary \ref{cor:one-end-plane}.

If $(g,k)=(1,1)$, then $\Sigma$ again has one embedded end, so Corollary \ref{cor:one-end-plane} yields that $\Sigma$ is a plane, which is impossible since the plane has genus zero. Hence this case cannot occur.

If $(g,k)=(0,2)$, then Theorem \ref{main-theorem} implies that $\Sigma$ is rotationally symmetric. Therefore $\Sigma$ is one of the special catenoids from Theorem \ref{existence}.

This proves the theorem.
\end{proof}

\begin{remark}\label{rem:sharpness}
Theorem \ref{thm:small-total-curvature} is sharp in the following sense. The special catenoids are complete connected embedded $f$-surfaces of genus zero with two embedded ends and finite total curvature, so by Theorem \ref{ThmJM} their total curvature is exactly $-4\pi$. Thus, the value $4\pi$ is attained and the first statement cannot be improved. As for the second statement, if $\int_\Sigma |K|\,dA=8\pi$ then $g+k=3$, and the arguments above exclude $(g,k)=(2,1)$ and $(g,k)=(1,2)$: the first by Corollary \ref{cor:one-end-plane} and the second by Theorem \ref{main-theorem}, since the special catenoids have genus zero. Hence the only configuration left open is $(g,k)=(0,3)$, that is, a genus-zero $f$-surface with three embedded ends and total curvature $-8\pi$. In the minimal case no such surface exists, by the L\'opez--Ros theorem \cite{Lopez-Ros}, which states that planes and catenoids are the only complete embedded minimal surfaces in $\mathbb{R}^3$ of genus zero and finite total curvature. A L\'opez--Ros type theorem for $f$-surfaces is not available (its proof relies on the L\'opez--Ros deformation, which has no known analogue in the Weingarten setting), and the existence of an $f$-surface with $(g,k)=(0,3)$ remains an interesting open problem.
\end{remark}

\section*{Acknowledgments}

The authors would like to thank J.A. Gálvez and P. Mira for their numerous thoughtful suggestions and improvements to the paper. The authors are also grateful to the anonymous referees of earlier versions of this work for their careful reading and valuable suggestions.

J.M. Espinar is partially supported by Spanish MIC Grant PID2024-160586NB-I00 and the {\it Maria de Maeztu} Excellence Unit IMAG, reference CEX2020-001105-M, funded by 
MCINN/AEI/10.13039/501100011033/CEX2020-001105-M.

H\'eber Mesa is supported by Universidad del Valle.

\nocite{Chern-Osserman}
\bibliography{bibliography-paper.bib}

\end{document}